\documentclass[10pt,bezier]{article}
\usepackage{amsmath,amssymb,amsfonts,graphicx,caption,subcaption}
\usepackage[colorlinks,linkcolor=red,citecolor=blue]{hyperref}

\newtheorem{prethm}{{\bf Theorem}}[section]
\newenvironment{thm}{\begin{prethm}{\hspace{-0.5
em}{\bf.}}}{\end{prethm}}
\newtheorem{prepro}{{\bf Theorem}}

\newtheorem{precor}[prethm]{{\bf Corollary}}
\newenvironment{cor}{\begin{precor}{\hspace{-0.5
em}{\bf.}}}{\end{precor}}
\newtheorem{preconj}[prethm]{{\bf Conjecture}}

\newtheorem{preremark}[prethm]{{\bf Remark}}
\newenvironment{remark}{\begin{preremark}\em{\hspace{-0.5
em}{\bf.}}}{\end{preremark}}
\newtheorem{prelem}[prethm]{{\bf Lemma}}
\newenvironment{lem}{\begin{prelem}{\hspace{-0.5
em}{\bf.}}}{\end{prelem}}
\newtheorem{preque}[prethm]{{\bf Question}}

\newtheorem{preobserv}[prethm]{{\bf Observation}}

\newtheorem{predef}[prethm]{{\bf Definition}}

\newtheorem{preproposition}[prethm]{{\bf Proposition}}
\newenvironment{proposition}{\begin{preproposition}{\hspace{-0.5
em}{\bf.}}}{\end{preproposition}}

\newtheorem{preproof}{{\bf Proof.}}
\newtheorem{preprooff}{{\bf Proof}}

\newenvironment{proof}[1]{\begin{preproof}{\rm
#1}\hfill{$\Box$}}{\end{preproof}}

\newtheorem{preproofs}{{\bf The second proof of }}

\newenvironment{proofs}[1]{\begin{preproofs}{\rm
#1}\hfill{$\Box$}}{\end{preproofs}}

\newtheorem{preprooft}{{\bf Third proof of }}

\newtheorem{preproofF}{{\bf Proof of}}

\newenvironment{proofF}[1]{\begin{preproofF}{\rm
#1}\hfill{$\Box$}}{\end{preproofF}}

\title{\bf\Large 
Packing spanning partition-connected subgraphs with small degrees
}
\author{{\normalsize{\sc Morteza Hasanvand${}$} }\vspace{3mm}
\\{\footnotesize{${}$\it Department of Mathematical
 Sciences, Sharif
University of Technology, Tehran, Iran}}
{\footnotesize{}}\\{\footnotesize{   $\mathsf{hasanvand@alum.sharif.edu  }$ }}}

\date{}
\def\OMEGA  {\text{$\Theta$}}
\begin{document}
\maketitle
\begin{abstract}{
Let $G$ be a graph  with $X\subseteq V(G)$ and let $l$ be an intersecting supermodular  subadditive integer-valued function on subsets of  $V(G)$. The graph $G$ is said to be $l$-partition-connected, if for every partition $P$ of $V(G)$, $e_G(P)\ge \sum_{A\in P} l(A)-l(V(G))$, where $e_G(P)$ denotes the number of edges of $G$ joining different parts of $P$. Let $\lambda \in [0,1]$ be a real number and let $\eta $ be a real function on $X$. In this paper, we show that if   $G$ is $l$-partition-connected  and for all $S\subseteq X$, $$\Theta_l(G \setminus S) \le \sum_{v\in S} (\eta(v) -2l(v))+l(V(G))+l(S)-\lambda (e_G(S))+l(S)),$$  then $G$ has an $l$-partition-connected spanning subgraph $H$ such that for each vertex $v\in X$,  $d_H(v)\le \lceil \eta(v) -\lambda l(v) \rceil $, where  $e_G(S)$ denotes the number of edges of $G$ with both ends in $S$ and  $\Theta_l(G \setminus S)$ denotes the maximum number  of all $\sum_{A\in P} l(A)-e_{G\setminus S}(P)$  taken over all partitions  $P$ of $V(G)\setminus S$.  Finally, we show that if $H$ is an  $(l_1+\cdots +l_m)$-partition-connected graph, then it can be decomposed into $m$ edge-disjoint spanning subgraphs $H_1,\ldots, H_m$ such that every graph $H_i$ is $l_i$-partition-connected, where  $l_1, l_2,\ldots, l_m$ are $m$  intersecting supermodular subadditive integer-valued functions on subsets of $V(H)$. These results generalize several known results.
} {\small
}
\end{abstract}
%
%
%==============================================================================
%
%
%
%
%
%
%
%
%
%
%
%
%									Introduction
%==============================================================================
\section{Introduction}
%==============================================================================
%									Definitions
In this article, all graphs have  no  loop, but  multiple  edges are allowed.
%and a simple graph is a graph without multiple edges.
%
%
 Let $G$ be a graph. 
The vertex set and  the edge set  of $G$ are denoted by $V(G)$ and $E(G)$, respectively. 
The degree $d_G(v)$ of a vertex $v$ is the number of edges of $G$ incident to $v$.
We denote by $d_G(C)$ the number of edges of $G$ with exactly one end in $V(C)$, where $C$ is a subgraph of $G$.
For a set $X\subseteq V(G)$, we denote by $G[X]$ the induced  subgraph of $G$  with the vertex set $X$  containing
precisely those edges 
of $G$  whose ends lie in~$X$.
For a spanning subgraph $H$  with the   integer-valued function $h$ on $V(H)$,
the {\bf total excess of $H$ from  $h$} is defined as follows:
$$te(H, h)=\sum_{v\in V(H)}\max\{0,d_H(v)-h(v)\}.$$
According to this definition, $te(H,h)=0$ if  and only if  for each vertex $v$, $ d_H(v)\le h(v)$.
Let $S\subseteq V(G)$. 
The vertex set $S$ is called   {\bf independent}, if there is no edge of $G$ connecting  vertices in $S$.
 The graph obtained from $G$ by removing all  vertices of $S$ is denoted by $G\setminus S$.
Let $F$ be a spanning subgraph of $G$.
Denote by $G\setminus [S,F]$  the graph obtained from $G$ by removing all edges incident to the vertices of $S$ except the edges of $F$. 
Note that while the vertices of $S$ are deleted in $G\setminus S$, no vertices are removed in $G\setminus [S, F]$.
Let $A$ and $B$ be two subsets of $V(G)$.
This pair   is said to be {\bf intersecting}, if $A\cap B \neq \emptyset$.
Let $l$ be a real function on subsets of $V(G)$ with $l(\emptyset) =0$.
For notational simplicity, we write $l(G)$ for $l(V(G))$ and write $l(v)$ for $l(\{v\})$.
The function $l$ is said to be {\bf supermodular}, if for all vertex sets $A$ and $B$,
$$l(A\cap B)+l(A\cup B)\ge l(A)+l(B).$$
Likewise, $l$ is said to be  intersecting  supermodular, if for all intersecting  pairs $A$ and $B$
the above-mentioned inequality holds.
The set function $l$ is called  (i) {\bf nonincreasing}, if  $l(A)\ge l(B)$, for all nonempty vertex sets $A,B$ with $A\subseteq B$,
(ii) {\bf subadditive}, if  $l(A) +l(B)\ge l(A\cup B)$, for any two disjoint   vertex sets $A$ and $B$,
(iii) {\bf element-subadditive}, if  $l(A)+l(v)\ge l(A\cup \{v\})$, for all  vertices $v$ and all vertex sets $A$ excluding $v$,
and also  is called (iv) {\bf weakly subadditive}, if  $\sum_{v\in A}l(v)\ge l(A)$, for all  vertex sets $A$.
Note that several  results of  this paper can be hold for real functions $l$ such that $\sum_{v\in A}l(v)-l(A)$ is integer for every vertex set $A$.
For clarity of presentation, we  will  assume that $l$ is  integer-valued.
The graph $G$ is said to be {\bf $l$-edge-connected}, if for all nonempty proper vertex sets $A$,
 $d_G(A) \ge l(A)$, where
$d_G(A)$ denotes the number of edges of $G$ with exactly one end in $A$.
Likewise, the graph $G$ is called {\bf $l$-partition-connected}, 
if for every partition $P$ of $V(G)$,
$e_G(P)\ge \sum_{A\in P}l(A)-l(V(G)),$
where $e_G(P)$  denotes the number of edges of $G$ joining different parts of $P$.
An $l$-partition-connected graph $G$ is {\bf minimally $l$-partition-connected}, 
if for every edge $e$ of $G$, the resulting $G-e$  is not $l$-partition-connected.
We will show that if $l$ is intersecting supermodular, then 
the
vertex set of $G$ can be expressed uniquely (up to order) as a disjoint union of vertex sets of some
induced $l$-partition-connected subgraphs.
These subgraphs are called the $l$-partition-connected components of $G$.
To measure $l$-partition-connectivity of $G$, we define the  parameter
$\OMEGA_l(G)=\sum_{A\in P} l(A)-e_G(P),$ where 
 $P$ is the partition of $V(G)$ obtained from $l$-partition-connected components of $G$.
The definition implies that for the null graph $K_0$ with no
vertices is $l$-partition-connected and
$\OMEGA_l(K_0 ) = 0$.
We will show that 
$\OMEGA_l(G)$ is  the maximum of all  $\sum_{A\in P} l(A)-e_G(P)$ taken over all
  partitions $P$ of $V(G)$.
We say that a spanning  subgraph $F$  is {\bf $l$-sparse}, if for all vertex sets $A$,
$e_F(A)\le \sum_{v\in A} l(v)-l(A)$, where 
 $e_F(A)$ denotes  the number of edges of $F$ with both ends in $A$. 
Clearly, $1$-sparse graphs are forests.
Note that  all    maximal $l$-sparse spanning subgraphs  of $G$ form  the bases of a matroid, 
when $l$ is an intersecting supermodular weakly subadditive integer-valued  function on subsets of $V(G)$,
 see~\cite{MR0270945}.
Note also that several  basic tools in this paper for working with sparse and partition-connected graphs can be obtained using matroid theory.
 A {\bf packing} refers to a collection of  edge-disjoint subgraphs.
A graph is said to be {\bf $m$-tree-connected}, it has $m$ edge-disjoint spanning trees.
It is known that every  $m$-partition-connected  graph is $m$-tree-connected~\cite{MR0133253, MR0140438}.
For every vertex set $A$ of a directed graph $G$,  we  denote  by $d^-_G(A)$  the number of edges   entering $A$ and denote by  $d^+_G(A)$ the number of edges  leaving  $A$.
An orientation of $G$ is called {\bf  $l$-arc-connected}, if for every vertex set $A$, $d_G^-(A)\ge l(A)$.
Likewise, an orientation of $G$ is called {\bf $r$-rooted $l$-arc-connected},
 if for every vertex set $A$, $d_G^-(A)\ge l(A)-\sum_{v\in A}r(v)$, where
 $r$ is a nonnegative integer-valued on $V(G)$ with $l(G)=\sum_{v\in V(G)}r(v)$.
 Throughout this article, we denote by $e^*_G(A)$
the maximum number of all $e_H(A)$ taken over all 
minimally $l$-partition-connected spanning subgraphs $H$ of $G$, all set functions are zero on the empty set, and also all  variables $k$ and $m$ are integer and positive, unless otherwise  stated.
%
%==============================================================================
%			
%
%
%
%
%
%
%
%
%==============================================================================
%		
%
%=-=-=-=-=-=-=-=-=

Recently, the present author~\cite{II} investigated bounded degree $m$-tree-connected spanning subgraphs
and established the following theorem.
This result gives a number of new applications on connected factors and
 generalizes and improves several known results in~\cite{MR1871346,MR1740929, MR519276, MR3336100, MR1621287, MR998275}.

\begin{thm}\label{intro:thm:general}{\rm (\cite{II})}
Let $G$ be an $m$-tree-connected graph  with $X\subseteq V(G)$.  
Let  $\lambda \in [0,1]$ be a real number and  let  $\eta$ be a  real function on $X$.
If  for all $S\subseteq X$,
$$\OMEGA_m(G\setminus S)\le\sum_{v\in S}\big(\eta(v)-2m\big)+2m-\lambda(e_G(S)+m),$$
then  $G$ has an $m$-tree-connected  spanning subgraph  $H$ such that for each $v\in X$,
 $d_H(v)\le \lceil \eta(v)-m\lambda\rceil$.
\end{thm}
In this paper, we   generalize the above-mentioned theorem to  the following supermodular version   
 by investigating  bounded degree  partition-connected spanning subgraphs.
Moreover, we generalize several results in~\cite{II}   toward this concept.
\begin{thm}\label{intro:thm:main}
Let $G$ be an $l$-partition-connected graph with $X\subseteq V(G)$, 
where  $l$ is an   intersecting supermodular  subadditive  integer-valued function on 
%*
 subsets of  $V(G)$. 
 Let  $\lambda \in [0,1]$ be a real number and  let  $\eta$ be a  real function on $X$.
If  for all $S\subseteq X$, 
$$\OMEGA_l(G\setminus S)\le \sum_{v\in S}\big(\eta(v)-2l(v)\big)+l(G)+l(S)-\lambda(e_G(S)+l(S)),$$
then  $G$ has an  $l$-partition-connected  spanning subgraph  $H$ such that for each $v\in X$, 
$d_H(v)\le \lceil \eta(v)-\lambda l(v) \rceil.$
\end{thm}

In Section~\ref{sec:packing}, we generalize the well-known result of 
Nash-Williams~\cite{MR0133253} and Tutte~\cite{MR0140438} to the following supermodular version.
This version can provide an alternative  proof for a special case of Theorem~\ref{intro:thm:main}.
\begin{thm}
{Let $H$ be a graph and let $l_1, l_2,\ldots, l_m$ be $m$  intersecting supermodular subadditive integer-valued functions on subsets of $V(H)$.  
 Then $H$ is $(l_1+\cdots +l_m)$-partition-connected, if and only if it can be decomposed into $m$ edge-disjoint 
spanning subgraphs $H_1,\ldots, H_m$ such that every  graph $H_i$ is $l_i$-partition-connected.
}\end{thm}
%
%

%
%==============================================================================
%
%
%
%
%
%
%
%
%==============================================================================
%
%
\section{Basic tools}
For every  vertex $v$ of a graph $G$, 
consider an  induced $l$-partition-connected subgraph of $G$ containing $v$ with the maximal order.
 The following proposition  shows that these subgraphs  are  unique and decompose the vertex set of $G$ when $l$
 is intersecting  supermodular.
In fact, these subgraphs are the $l$-partition-connected components of G that already introduced in the Introduction.
\begin{proposition}\label{thm:basic:union}
{Let $G$ be a graph with $X,Y\subseteq V(G)$ and let $l$ be an intersecting supermodular  real function   on   subsets of $V(G)$.
If  $G[X]$ and $G[Y]$ are $l$-partition-connected and $X\cap Y\neq \emptyset$, 
then $G[X\cup Y]$ is also $l$-partition-connected.
}\end{proposition}
\begin{proof}
{Let $P$ be a partition of $X\cup Y$.
Take $A_1,\ldots, A_n$  to be all vertex sets belonging to $P$ such that for each $i$ 
with $1\le i \le n$, $A_i\cap X\neq \emptyset$ and $A_i\cap Y\neq \emptyset$.
Set $A_{n+1} =Y$.
Let $P_1$  be the set of all vertex sets  $A\in P$ with $A\subseteq X\setminus Y$, 
and set $P'_1=P_1\cup \{A_i\cap X: 1\le i \le n\}$.
Let $P_2$ be the set of all vertex sets   $A\in P$ with $A\subseteq Y\setminus X$, and set $P'_2=P_2\cup \{Z\}$,
 where  $X\cap Y\subseteq Z=(\cup_{1\le i\le n} A_i)\cap Y$.
Define $B_1=X$ and for every positive integer $i$ with $1 \le i \le n+1$ recursively define $B_i=B_{i-1}\cup A_{i-1}$.
Note that $B_i\cap A_i \neq \emptyset$.
It is easy to  check that
$$
e_{G[X\cup Y]}(P)\ge e_{G[X]}(P'_1)+ e_{G[Y]}(P'_2)+
\sum_{1\le i\le n+1} d_G(B_i, A_i),
$$
where  $d_G(B_i,A_i)$ denotes the number of edges of $G$ with one end in $B_i\setminus A_i$ and other one in $A_i\setminus B_i$.
Since $G[X]$ and $G[Y]$ are $l$-partition-connected,
$$e_{G[X\cup Y]}(P)\ge  
\sum_{A\in P'_1} l(A)-l(X)+\,
\sum_{A\in P'_2} l(A)-l(Y)+\,
\sum_{1\le i\le n+1} d_G(B_i, A_i).$$
which implies that
\begin{equation}\label{eq:convex:1}
e_{G[X\cup Y]}(P)\ge 
\sum_{A \in P_1} l(A)+
\sum_{1\le i\le n} l(A_i \cap X)-l(X)+
\sum_{A\in P_2} l(A)+l(Z)-l(Y)+
\sum_{1\le i\le n+1} d_G(B_i, A_i).
\end{equation}
By the assumption, for each $i$ with $1 \le i\le  n+1$, we have 
$$l(B_{i}\cap A_i)+l(B_{i}\cup A_i)\, + d_G(B_{i},A_i)\ge l(B_{i})+l(A_i),$$
which implies that
\begin{equation}\label{eq:convex:2}
\sum_{1\le i\le n} l(X\cap A_i)+l(Z)
+l(X\cup Y)+\sum_{1\le i\le n +1}
d_G(B_i,A_i)\ge
l(X)+
 \sum_{1\le i\le n}l(A_i) +l(Y).
\end{equation}
Therefore, Relations~(\ref{eq:convex:1}) and~(\ref{eq:convex:2}) can conclude that
$$e_{{G[X\cup Y]}}(P)
\ge
\sum_{ A\in P_1} l(A)
+\sum_{ 1\le i \le n} l(A_i)-l(X\cup Y)+
\sum_{ A\in P_2} l(A) =
\sum_{A\in P} l(A)-l(X\cup Y).$$
Hence the proposition holds.
}\end{proof}
The next proposition presents a simple way for deducing partition-connectivity of a graph
from whose contractions and whose special subgraphs.
\begin{proposition}\label{prop:deducing}
{Let $G$ be a graph with $X\subseteq V(G)$ and let $l$ be an intersecting supermodular real  function   on
  subsets of $V(G)$.
If $G[X]$ and $G/X$ are $l$-partition-connected, then $G$ itself is $l$-partition-connected.
}\end{proposition}
\begin{proof}
{It is enough to apply the same arguments in the proof of Theorem~\ref{thm:basic:union}, by setting $Y=V(G)$.
Note that  we still  have $e_G(P'_2) \ge \sum_{A\in P'_2} l(A)-l(Y)$, since $G/X$ is $l$-partition-connected.
}\end{proof}
\subsection{Minimally partition-connected and  maximal sparse spanning subgraphs}
The following lemma presents a simple way for inducing $l$-partition-connectivity of a graph
to whose special subgraphs. 
\begin{lem}\label{lem:inducing}
{Let $G$ be a graph and let $l$ be a real  function on
%*
subsets of $V(G)$.
If $G$ is $l$-partition-connected and $P$ is a   partition of $V(G)$ with 
$$e_G(P)=\sum_{A\in P}l(A)-l(G),$$
then for any $A\in P$, the graph $G[A]$ is also $l$-partition-connected.
}\end{lem}
\begin{proof}
{Let $A\in P$
and let $P'$ be an arbitrary partition of $A$.
Define $P''$ to be the partition of $V(G)$  with $P''=P'\cup (P-A)$. 
Since $G$ is $l$-partition-connected, 
$$e_{G[A]}(P') =e_{G}(P'') -e_{G}(P) \ge
 \sum_{A'\in P''}l(A')-l(G)- (\sum_{A'\in P}l(A')-l(G))=
\sum_{A'\in P'} l(A')-l(A).$$
Hence $G[A]$ is also $l$-partition-connected.
}\end{proof}
The following proposition establishes a simple but important property of minimally partition-connected graphs.
\begin{proposition}\label{prop:minimally-edges}
{Let $H$ be a graph and let  $l$ be  an intersecting supermodular weakly subadditive integer-valued function on 
%*
subsets of $V(H)$.
If  $H$ is minimally $l$-partition-connected, then
$$|E(H)|= \sum_{v\in V(G)} l(v)-l(H).$$
}\end{proposition}
\begin{proof}
{By induction on $|V(H)|$. For $|V(H)|=1$ the proof is clear. So, suppose $|V(H)|\ge 2$.
Since $H$ is $l$-partition-connected and $l$ is weakly subadditive, we have
$|E(H)|\ge \sum_{v\in V(H)} l(v)-l(H)\ge 0.$
If  $|E(H)|=0$, then the theorem holds.
So, suppose that $|E(H)| >0$ and 
let $e$ be a fixed edge of $H$.
Since $H-e$ is not $l$-partition-connected, there is a partition $P$ of $V(H)$ such that  
$e_H(P)=\sum_{A\in P} l(A)-l(H)$ and $e$ joins different parts of $P$.
By Lemma~\ref{lem:inducing}, for every $A\in P$, the $H[A]$ is $l$-partition-connected.
If for an edge $e'\in E(H[A])$, $H[A]- e'$ is still $l$-partition-connected, then 
 by Proposition~\ref{prop:deducing}, one  can conclude that $H\setminus e'$ is  still $l$-partition-connected, which is impossible. 
Thus $H[A]$ is minimally $l$-partition-connected and by  induction hypothesis, we therefore
have 
$$|E(H)|=\sum_{A\in P}e_H(A)+e_H(P)=
\sum_{A\in P}(\sum_{v\in A} l(v) -l(A))+\sum_{A\in P} l(A) -l(H)=
\sum_{v\in V(H)} l(v)-l(H),$$
 which  completes the proof.
}\end{proof}
The following proposition shows that maximal  sparse  spanning  graphs are also  partition-connected.
\begin{proposition}\label{prop:sparse}
{Let $F$ be an $l$-sparse graph with $|E(F)| = \sum_{v\in V(F)}l(v)-l(F)$, where $l$ is a weakly subadditive real function on 
%*
subsets of $V(F)$. 
If $P$ is a partition of $V(F)$, then
$$e_F(P)\ge \sum_{A\in P}l(A)-l(F).$$
Furthermore, the equality holds only if for every $A\in P$, the graph $F[A]$ is $l$-partition-connected.
}\end{proposition}
\begin{proof}
{Since $F$ is $l$-sparse,   $e_F(A)\le \sum_{v\in A}l(v)-l(A)$, for every $A\in P$, which implies that
$$e_F(P)=|E(F)|-\sum_{A\in P}e_F(A)\ge  \sum_{v\in V(F)}l(v)-l(F)-\sum_{A\in P}(\sum_{v\in A}l(v)-l(A)) =\sum_{A\in P} l(A)-l(F).$$
Furthermore, if the equality holds, then for every $A\in P$, we must have   $e_F(A)= \sum_{v\in A}l(v)-l(A)$.
Since the induced  graph $F[A]$ is $l$-sparse, it must be $l$-partition-connected.
Hence the proof is completed.
}\end{proof}
\begin{proposition}\label{prop:Q:subadditive}
{Let $F$ be an $l$-sparse graph with $x,y\in V(F)$, where  $l$ is  a weakly subadditive real function on 
%*
subsets of $V(F)$.
Let   $Q$ be an  $l$-partition-connected subgraph of $F$ with the minimum number of vertices including $x$ and $y$.
If $l$ is element-subadditive, then for each $z\in V(Q)\setminus \{x,y\}$, $d_Q(z)\ge 1$.
Furthermore, if $l$ is subadditive, then for every vertex set $A$ with $\{x,y\}\subseteq A\subsetneq V(Q)$, $d_Q(A)\ge 1$.
}\end{proposition}
\begin{proof}
{Let $A$ be a vertex set with $\{x,y\}\subseteq A\subsetneq V(Q)$ and set $B=V(Q)\setminus A$.
According to the minimality of $Q$,  the graph $Q[A]$
 is not partition-connected and so by Proposition~\ref{prop:sparse}, we must have
$d_Q(B)=d_Q(A)> l(A)+l(B)-l(A\cup B)\ge 0$, whether
 $l$ is element-subadditive and  $B=\{z\}$ or
$l$ is subadditive.
}\end{proof}
%
%
%
%
%
%
%
%-------------------------------------------------------------------------------
%
\subsection{Exchanging edges and preserving partition-connectivity}
The following proposition is a  useful tool for finding a pair of edges such that replacing
them preserves partition-connectivity of a given spanning subgraph. 
\begin{proposition}\label{prop:replacing}
{Let $G$ be a graph and let $l$ be an  intersecting supermodular  integer-valued  function on
%*
subsets of $V(G)$.
Let $H$ be an  $l$-partition-connected spanning subgraph of $G$ and 
let $M$ be a nonempty edge subset of $E(H)$.
If a given  edge $e'\in E(G)\setminus E(H)$  joins different $l$-partition-connected components of $H\setminus M$, 
then there is an  edge 
$e$ belonging to $ M$ such that $H-e+e'$ is still $l$-partition-connected.
}\end{proposition}
\begin{proof}
{We proceed
by induction on $|M|$. 
Assume first that  $M=\{e\}$.
Suppose, by way of contradiction,  that $H-e+e'$ is not $l$-partition-connected.
Consequently, there is a partition $P$
 of $V(H')$ such that  $e_{H'}(P)<   \sum_{A\in P}l(A)-l(G)$, where $H'=H-e+e'$. 
Since   $l$ is integer-valued,  $ \sum_{A\in P}l(A)-l(G)$  is  integer, and so  
$e_{H'}(P)\le    \sum_{A\in P}l(A)-l(G)-1$.
Since $H$ is $l$-partition-connected, we must have  
$ e_{H'}(P) =e_{H}(P)-1$ and  $e_{H}(P)=    \sum_{A\in P}l(A)-l(G)$. 
Therefore,
the edge $e$  joins     different parts of $P$ and 
both ends of $e'$ lie in the same part $A$ of $P$.
By Lemma~\ref{lem:inducing}, the graph $H[A]$ is $l$-partition-connected, which is a contradiction.
Now, assume  that $|M|\ge 2$.
Pick $e\in M$.
If   $e'$ whose ends lie in different $l$-partition-connected components of $H-(M\setminus e)$, 
then the proof follows  by induction.
Suppose that both ends of $e'$  lies in the same  $l$-partition-connected component $C$ of $H-(M\setminus e)$.
By the assumption, both ends of  $e$ must  lie in $C$ and also $e'$ whose ends lies 
in different $l$-partition-connected components of $C-e$.
By applying induction  to $C$, the graph $C-e+e'$ must be $l$-partition-connected.
Thus by  Proposition~\ref{prop:deducing},
 the graph $H-e+e'$  is  $l$-partition-connected.
Hence the proposition holds.
}\end{proof}
The next proposition    is a useful  tool for finding a pair of edges such that replacing
them preserves sparse property  of a given sparse spanning subgraph.  
\begin{proposition}\label{prop:xGy-exchange}
{Let $G$ be a graph and let $l$ be an  intersecting supermodular  weakly subadditive  integer-valued function on 
%*
subsets of $V(G)$.
Let $F$ be an $l$-sparse spanning subgraph  of $G$.  
If  $xy\in E(G)\setminus E(F)$ and $Q$ is an  $l$-partition-connected subgraph of $F$ including $x$ and $y$ with the minimum number of vertices,  then for every  $e \in E(Q)$, the graph $F-e+xy$ remains $l$-sparse.
}\end{proposition}
\begin{proof}
{If  $F-e+xy$  is not $l$-sparse,
 then there is a vertex set $A$ including $x$ and $y$ such that $e\notin E(F[A])$ and $e_F(A)=\sum_{v\in A}l(v)-l(A)$. 
  Since $F$ is $l$-sparse,
$$e_F(A\cap B)\ge e_F(A)+e_F(B)-e_F(A\cup B)\ge
\sum_{v\in A}l(v)-l(A)+\sum_{v\in B}l(v)-l(B)-
 \sum_{v\in A\cup B}l(v)+l(A\cup B).$$
where  $B=V(Q)$.
Since $l$ is intersecting supermodular, we therefore,  $e_F(A\cap B)\ge \sum_{v\in A\cap B}l(v)+l(A\cap B)$.
By Proposition~\ref{prop:sparse}, the graph $F[A\cap B]$ must be $l$-partition-connected, which contradicts minimality of $Q$.
Note that $F[A\cap B]$ includes $x$ and $y$.
Hence the the proof is completed.
}\end{proof}
\subsection{Comparing partition-connectivity measures}
The following lemma  
 gives useful information about the existence of non-trivial $l$-partition-connected components and
 develops a result  in~\cite{MR2575822}.
\begin{lem}\label{lem:non-trivial-component}
{Let $G$ be a graph of order at least two and let $l$ be a real  function on subsets of $V(G)$.
If  $G$   contains at least $\sum_{v\in V(G)}l(v)-l(G)$ edges, then it  has an
$l$-partition-connected subgraph with at least two vertices.
}\end{lem}
\begin{proof}
{The proof is by induction on $|V(G)|$. 
For $|V(G)|=2$, the proof is clear.
Assume $|V(G)|\ge 3$.
Suppose the lemma is false. 
Thus there exists a partition $P$ of $V(G)$ such that
 $e_G(P)< \sum_{A\in P}l(A)-l(G)$.
By induction hypothesis, for every  $A\in P$, 
we have $e_{G}(A)\le\sum_{v\in A}l(v)-l(A)$, whether $|A|\ge 2$ or not. 
Therefore, 
$$ \sum_{v\in V(G)}l(v)-l(G) \le |E(G)|= e_G(P)+ \sum_{A\in P}e_G(A)
<
 \sum_{A\in P}l(A)-l(G)+
\sum_{A\in P}(\sum_{v\in A}l(v)-l(A))
= \sum_{v\in V(G)}l(v)-l(G).$$
This result is a contradiction, as desired.
}\end{proof}
The following result describes a relationship between  partition-connectivity measures of graphs.
\begin{thm}
{Let $G$ be a graph and let $l$ an intersecting supermodular real function on 
%*
subsets of $V(G)$.
If  $\beta$ is a real number with  $\beta\ge 1$, then
$$l(G)\le \OMEGA_{l}(G) \le \frac{1}{\beta}\OMEGA_{\beta l}(G).$$
Furthermore, $G$ is $l$-partition-connected if and only if $\OMEGA_l(G)=l(G)$.
}\end{thm}
\begin{proof}
{Define $l'=\beta l$. Note that $l'$ is also intersecting supermodular.
Let $P$ and $P'$  be the partitions of $V(G)$ obtained from the $l$-partition-connected components and 
$l'$-partition-connected components  of $G$. If $G$ is $l$-partition-connected, 
then  we have $|P|=1$ and so $e_G(P)=0$ and  $\OMEGA_l(G)=l(G)$.
Oppositely, if $G$ is not $l$-partition-connected, then by applying  Lemma~\ref{lem:non-trivial-component} to the contracted graph $G/P$,
 $e_G(P)<\sum_{A\in P}l(A)-l(G)$ and hence $\OMEGA_l(G)> l(G)$.
For every $X\in P$, define $P'_X$ to be the partition of $X$ obtained from the vertex sets of $P'$.
By applying  Lemma~\ref{lem:non-trivial-component} to the graph $G[X]/P'_X$, we  have
 $e_{G[X]}(P'_X) \le \sum_{A\in P'_X}l'(A)-l'(X)$, whether $|P'_X|=1$ or not.
Therefore, 
$$\OMEGA_{l'}(G)=\sum_{A\in P'}l(A)-e_G(P')=
\sum_{X\in P}(\sum_{A\in P'_X} l'(A)-e_{G[X]}(P'_X))\,-e_G(P)\ge 
\sum_{X\in P}l'(X)-e_G(P)\ge \beta
\OMEGA_{l}(G).$$
This equality completes the proof.
}\end{proof}
The following theorem introduces an  interesting property of partition-connectivity measures.
\begin{thm}
{Let $G$ be a graph and  let $l$ be an intersecting supermodular real function on 
%*
subsets of $V(G)$.
Then we have,
$$ \OMEGA_l(G) =\max \big\{\sum_{A\in P}l(A)-e_G(P): P \text{ is a partition of $V(G)$}\big\}.$$
}\end{thm}
\begin{proof}
{Consider $P$ with the maximum 
$\sum_{A\in P}l(A)-e_G(P)$ and with the minimal $|P|$.
 If  for a vertex set $X\in P$, the graph  $G[X]$ is not $l$-partition-connected, 
then there is a partition $P'$ of $X$ such that  $e_{G[X]}(P')< \sum_{A\in P'} l(A) -l(X).$
Define $P''=P'\cup (P-X)$.
Then
$$\sum_{A\in P''}l(A)-e_G(P'')=\sum_{A\in P}l(A)-l(X)-e_{G}(P)+\sum_{A\in P'}l(A)-e_{G[X]}(P')> 
\sum_{A\in P}l(A)-e_G(P),$$
which contradicts maximality of $\sum_{A\in P}l(A)-e_G(P)$.
Hence    for every  set $X\in P$, the graph  $G[X]$ must be $l$-partition-connected.
Now, assume that   $G[X']$ is $l$-partition-connected, where 
$X'=\cup_{A\in P'} A$, $P'\subseteq P$, and $|P'| \ge 2$.
Thus $e_G[X'](P')\ge \sum_{A\in P'} l(A) -l(X')$.
Define $P''=(P\setminus P')\cup \{X'\}$. Then
$$\sum_{A\in P''} l(A)-e_G(P'')=\sum_{A\in P} l(A)+l(X')-\sum_{A\in P'} l(A)-(e_G(P)-e_{G[X']}(P'))\ge
 \sum_{A\in P}l(A)-e_G(P),$$
which contradicts minimality  of $|P|$.
It  is easy to check that $P$ must be the same partition of $G$ obtained from $l$-partition-connected components of $G$.
Hence the theorem holds.
}\end{proof}
%
%
%
%
%
%
%
%
%
%
%
%==============================================================================
%
%
%
%
%
%
%
%
%==============================================================================
%
%
%
\section{Highly partition-connected spanning subgraphs with small  degrees}

Here, we state following fundamental theorem, which gives much information about partition-connected spanning subgraphs  with the minimum total excess. 
In Section~\ref{sec:Generalizations},
 we present a stronger version for  this result with a proof, but we feel that it helpful to state the
proof of this special case before the general version.
\begin{thm}\label{thm:preliminary:structure}
{Let $G$ be a graph, let $l$ be an  intersecting supermodular element-subadditive  integer-valued function on
%*
subsets of $V(G)$,
 and 
 let  $h$ be an  integer-valued function on $V(G)$.
 If $H$ is  a minimally   $l$-partition-connected spanning subgraph  of $G$  with the minimum total excess from $h$,  
then there exists a subset $S$ of $V(G)$ with the following properties:
\begin{enumerate}{
\item $\OMEGA_l (G\setminus S)=\OMEGA_l (H\setminus S)$.\label{Condition 2.1}
\item $S\supseteq \{v\in V(G):d_H(v)> h(v)\}$.\label{Condition 2.2}
\item For each vertex $v$ of $S$, $d_H(v)\ge h(v)$.\label{Condition 2.3}
}\end{enumerate}
}\end{thm}
%
%%%%%%
\begin{proof}
{Define $V_0=\emptyset $ and  $V_1=\{v\in V(H): d_H(v)> h(v)\}$.
For any $S\subseteq V(G)$ and   $u\in V(G)\setminus S$, 
let $\mathcal{A}(S, u)$ be the set of all minimally  $l$-partition-connected spanning subgraphs of
$H^\prime $ of $G$ 
such that    $d_{H'}(v)\le h(v)$ for all $v\in V(G)\setminus V_1$,    and 
$H^\prime$ and $H$ have the same edges, except for some of the edges of $G$ whose ends are in $X$, 
where $H[X]$ is the  $l$-partition-connected component of  $H\setminus S$ containing $u$.
Note that $H'[X]$   must automatically be   $l$-partition-connected.
Now, for each integer $n$ with $n\ge 2$,   recursively define $V_n$ as follows:
$$V_n=V_{n-1} \cup \{\, v\in V(G)\setminus V_{n-1} \colon \, d_{H^\prime }(v)\ge h(v) \text{,\, for all\, }H^\prime \in \mathcal{A}(V_{n-1},v)\,\}.$$
 Now, we prove the following claim.
%
%===================================	Claim
\vspace{2mm}\\
{\bf Claim.} 
Let $x$ and $y$ be two vertices in different  $l$-partition-connected components  of $H\setminus V_{n-1}$.
If $xy\in E(G)\setminus E(H)$, then $x\in V_{n}$ or $y\in V_{n}$.
\vspace{2mm}\\
%===================================
{\bf Proof of Claim.} 
By induction on $n$.
 For $n=1$, the proof is clear. 
Assume that the claim is true for $n-1$.
Now we prove it for $n$.
Suppose otherwise that vertices $x$ and $y$ are in different  $l$-partition-connected components of $H\setminus V_{n-1}$, respectively, with the vertex sets $X$ and $Y$,
$xy\in E(G)\setminus E(H)$, 
and $x,y\not \in V_{n}$. 
Since $x,y\not\in  V_{n}$,
 there exist
 $H_x\in \mathcal{A}(V_{n-1},x)$ and
 $H_y\in \mathcal{A}(V_{n-1},y)$ with $d_{H_x}(x)< h(x)$ and $d_{H_y}(y)< h(y)$. 
By the induction hypothesis,
 $x$ and $y$ are in the same $l$-partition-connected components
% Z
of $H\setminus V_{n-2}$ with the vertex set $Z$ so that $X\cup Y \subseteq Z$.
Let $Q$ be the unique $l$-partition-connected subgraph of $H$  with minimum number of vertices including $x$ and $y$.
 Notice that the vertices of $Q$ lie in $Z$ and also  $Q$ includes at least a vertex $z$ of $Z\cap V_{n-1}$ so that $d_H(z)\ge h(z)$.
Since $l$ is element-subadditive, $d_Q(z)\ge 1$ which  means  that  there is
  an edge $zz'$ of $Q$ incident to $z$. 
By Proposition~\ref{prop:replacing}, 
 the graph $H[Z]-zz'+xy$ must be  
$l$-partition-connected.
Now, let  $H'$ be the spanning subgraph of $G$ with
 $$E(H')=E(H)-zz'+xy
-E(H[X])+E(H_x[X])
-E(H[Y])+E(H_y[Y]).$$
By  repeatedly  applying Proposition~\ref{prop:deducing},  one can easily check that $H'$  is $l$-partition-connected. 
For each  $v\in V(H')$, we  have
$$d_{H'}(v)\le   
 \begin{cases}
d_{H_v}(v)+1,	&\text{if   $v\in \{x,y\}$};\\
d_{H}(v),	&\text{if   $v=z'$},
\end {cases}
\quad \text{ and }\quad
 d_{H'}(v)= 
 \begin{cases}
d_{H_x}(v),	&\text{if   $v\in X\setminus \{x,z'\}$};\\
d_{H_y}(v),	&\text{if   $v\in Y\setminus \{y,z'\}$};\\
d_{H}(v),	&\text{if   $v\notin X\cup Y\cup \{z,z'\}$}.
\end {cases}$$
If $n\ge 3$, then it is not hard to see that  $d_{H'}(z)<d_{H}(z)\le h(z)$ and $H'$ lies in $ \mathcal{A}(V_{n-2}, z)$.
Since $z\in V_{n-1}\setminus V_{n-2}$,  we arrive at a contradiction.
For the case $n=2$,  
since $z\in V_1$, it is easy to see that
$h(z)\le d_{H'}(z)<d_{H}(z)$ and  $te(H',h)< te(H,h)$,   which is again a contradiction.
Hence the claim holds.

%
%=====================================  End Claim
%
%
Obviously, there exists a positive integer $n$ such that  $V_1\subseteq \cdots\subseteq V_{n-1}=V_{n}$.
 Put $S=V_{n}$.  
Since $S\supseteq V_1$, Condition~\ref{Condition 2.2} clearly holds.
For each $v\in V_i\setminus V_{i-1}$ with $i\ge 2$, 
we have $H\in \mathcal{A}(V_{i-1},v)$ and so  $d_H(v)\ge h(v)$. 
This establishes   Condition~\ref{Condition  2.3}.
 Because $S=V_{n}$, 
the previous claim implies Condition~\ref{Condition 2.1} and  completes the proof.
}\end{proof}
\begin{remark}
{The element-subadditive condition of Theorem~\ref{thm:preliminary:structure}  can be replaced  by weakly subadditive condition along with  the conditions
$l(A)+l(v)\ge l(A\cup \{v\})$, where $v$ is a vertex with  $ h(v) \le d_G(v)$  and $A\subsetneq V(G)\setminus v$.
This version allows us to construct another appropriate set function only by increasing $l(G)$. 
}\end{remark}
\begin{remark}
{Note that the proof of Theorem~\ref{thm:preliminary:structure}
  can be modified to present  a  polynomial-time algorithm,  similar to the algorithm of the proof of Theorem~1 in~\cite{MR1871346}, for producing an appropriate vertex set  $S$; by considering calling of  subroutines for finding partition-connected components and minimal partition-connected subgraphs as  single steps.
}\end{remark}
%
%
%==============================================================================
%
%
%
%
%
%
%
%
%
%
%==============================================================================
%
%
%
%
\subsection{Sufficient conditions depending on partition-connectivity measures}
The following lemma establishes an important property of  minimally $l$-partition-connected graphs.
\begin{lem}\label{lem:minimally-omega}
{Let $H$ be a graph  and let $l$ be an intersecting supermodular weakly subadditive integer-valued function on
%*
subsets of $V(H)$.
If  $H$ is minimally $l$-partition-connected  and  $S\subseteq V(H)$,
 then $$\OMEGA_l (H\setminus S)= \sum_{v\in S}(d_H(v)-l(v))+l(H)-e_H(S).$$
}\end{lem}
\begin{proof}
{Let $P$ be the  partition of $V(H)\setminus S$ obtained from the  
  $l$-partition-connected  components of $H\setminus S$.
Obviously, $e_H( P\cup \{\{v\}:v\in S\})=\sum_{v\in S}d_H(v)\,-e_H(S)+e_{H\setminus S}(P)$.
By Proposition~\ref{prop:deducing}, one can conclude that $H[A]$ is minimally $l$-partition-connected, for any $A\in P$.
Hence  Proposition~\ref{prop:minimally-edges} implies that 
 $e_H(A)= \sum_{v\in A}l(v)-l(A)$, and also $|E(H)|=\sum_{v\in V(H)}l(v)-l(H)$.
Thus
$$e_H( P\cup \{\{v\}:v\in S\})=
|E(H)|-\sum_{A\in P} e_H(A)=
 \sum_{v\in S}l(v)+
\sum_{A\in P} l(A)-l(H).$$
Therefore,
$$\OMEGA_l (H\setminus S)=\sum_{A\in P}l(A)-e_{H\setminus S}(P)= \sum_{v\in S}(d_H(v)-l(v))+l(H)-e_H(S).$$
Hence the lemma is proved.
}\end{proof}
The following theorem is essential in this section.
\begin{thm}\label{thm:sufficient}
Let $G$ be a graph with $X\subseteq V(G)$  and let $l$ be an intersecting supermodular element-subadditive integer-valued function on 
%*
 subsets of $V(G)$. 
 Let  $\lambda \in [0,1]$ be a real number and  let  $\eta$ be a  real function on $X$.
If  for all $S\subseteq X$, 
$$\OMEGA_l(G\setminus S)< 1+\sum_{v\in S}\big(\eta(v)-2l(v)\big)+l(G)+l(S)-\lambda(e^*_G(S)+l(S)),$$
then  $G$ has an  $l$-partition-connected  spanning subgraph  $H$ such that for each $v\in X$, 
$d_H(v)\le \lceil \eta(v)-\lambda l(v)\rceil.$
\end{thm}
%
%
%---------------------------------------------------------------
\begin{proof}
{For each vertex $v$, define 
$$h(v)= 
 \begin{cases}
d_G(v)+1,	&\text{if   $v\not\in X$};\\
\lceil \eta(v)-\lambda l(v)\rceil,	&\text{if $v\in X$}. 
\end {cases}$$
%---------------------------------------------------------------
Note that $G$ is automatically $l$-partition-connected, because  of $\OMEGA_l(G\setminus \emptyset) \le l(G)$.
Let $H$ be a minimally $l$-partition-connected  spanning subgraph  of $G$ with the minimum total excess from $h$.
Define $S$ to be a subset of $V (G)$  with the properties described in 
Theorem~\ref{thm:preliminary:structure}.
Obviously, $S\subseteq X$.
By Lemma~\ref{lem:minimally-omega},  
$$
\sum_{v\in S}h(v)\, +te(H,h)=
 \sum_{v\in S}d_H(v)= \OMEGA_l(H\setminus S)+\sum_{v\in S}l(v)-l(H)+e_H(S).
$$
and so
\begin{equation}\label{eq:A:1}
\sum_{v\in S}h(v)\, +te(H,h)=  \OMEGA_l(G\setminus S)+\sum_{v\in S}l(v)-l(G)+e_H(S).
\end{equation}
Also, by the assumption, we  have
\begin{equation}\label{eq:A:2}
\OMEGA_l(G\setminus S)+\sum_{v\in S}l(v)-l(G)+e_H(S)<
 1+\sum_{v\in S}\big(\eta(v)-l(v)\big)-\lambda(e^*_G(S)+l(S))+e_H(S)+l(S).
\end{equation}
Since
  $e_H(S)\le e^*_G(S)$ and 
$e_H(S)
\le \sum_{v\in S}l(v) -l(S)$,  
\begin{equation}\label{eq:A:3}
-\lambda(e^*_G(S)+l(S))+e_H(S)+l(S)\le  -\lambda(e_H(S)+l(S))+e_H(S)+l(S)\le (1-\lambda)\sum_{v\in S}l(v).
\end{equation}
Therefore, Relations (\ref{eq:A:1}),  (\ref{eq:A:2}), and (\ref{eq:A:3})  can conclude that
$$\sum_{v\in S}h(v)+te(H,h)<
1+ \sum_{v\in S}\big(\eta(v)-\lambda l(v) \big).$$
On the other hand, by the definition of $h(v)$, 
$$ \sum_{v\in S}\big(\eta(v)-\lambda l(v)-h(v)\big)  \le 0.$$
Hence  $te(H,h) = 0$ and  the theorem holds.
}\end{proof}
 When we consider the special cases $\lambda = 1$, the theorem becomes simpler as the following result.
\begin{cor}\label{cor:case1}
Let $G$ be a graph  and let $l$ be an intersecting supermodular element-subadditive integer-valued function on 
%*
 subsets of $V(G)$. 
Let  $h$ be an  integer-valued function on $V(G)$.
If  for all $S\subseteq V(G)$, 
$$\OMEGA_l(G\setminus S)\le \sum_{v\in S}\big(h(v)-l(v)\big)+l(G)-e^*_G(S),$$
then  $G$ has an $l$-partition-connected  spanning subgraph  $H$ such that for each vertex  $v$,
 $d_H(v)\le h(v).$
\end{cor}
\begin{proof}
{Apply Theorem~\ref{thm:sufficient} with $\lambda = 1$ and $\eta(v) = h(v)+l(v)$.
}\end{proof}
Note that the above-mentioned corollary is equivalent to Theorem~\ref{thm:sufficient} and can concludes the next results.
\begin{cor}
Let $G$ be a graph  and let $l$ be an intersecting supermodular element-subadditive integer-valued function on subsets of $V(G)$. 
Let  $h$ be an  integer-valued function on $V(G)$.
If  for all $S\subseteq V(G)$, 
$$\OMEGA_l(G\setminus S)\le \sum_{v\in S}\big(h(v)-2l(v)\big)+l(G)+\OMEGA_l(G[S]),$$
then  $G$ has an $l$-partition-connected  spanning subgraph  $H$ such that for each vertex  $v$,
 $d_H(v)\le h(v).$
\end{cor}
\begin{proof}
{Apply Corollary~\ref{cor:case1} along with the inequality 
$e^*_G(S) \le   \sum_{v\in S} l(v)-\OMEGA_l(G[S]).$
}\end{proof}
The following corollary provides a necessary and sufficient condition for the existence of a partition-connected spanning subgraph with
the described properties.
\begin{cor}
{Let $G$ be a graph with independent set $X\subseteq V(G)$  and let $l$ be an intersecting supermodular element-subadditive integer-valued function on 
%*
subsets of $V(G)$. 
Let $h$ be an integer-valued  function on $X$.
Then  for all $S\subseteq X$, 
$$\OMEGA_l(G\setminus S)\le \sum_{v\in S}\big(h(v)-l(v)\big)+l(G),$$
if and only if $G$ has an $l$-partition-connected spanning subgraph $H$ 
such that for each $v\in X$, $d_H(v)\le h(v)$.
}\end{cor}
\begin{proof}
{It is enough to apply Corollary~\ref{cor:case1} with $e^*_G(S) = 0$, and apply Lemma~\ref{lem:minimally-omega}. 
}\end{proof}
%%
%
%
%
%
%=============================================================================
%
%
%
%
%
%
%
%
%
%
%=============================================================================
%
\subsection{An alternative proof for a weaker version  of Corollary~\ref{cor:case1}}
In this  subsection, we are going to present another proof for the following weaker version of Corollary~\ref{cor:case1}.
Our proof is based on orientations of partition-connected graphs. In Section~\ref{sec:hypergraphs}, 
we alternatively  present a new proof for it based on edge-decompositions with a stronger version on hypergraphs.
\begin{thm}\label{thm:alternative}
Let $G$ be a graph  and let $l$ be a nonincreasing intersecting supermodular nonnegative integer-valued function on 
subsets of $V(G)$. 
 Let $h$ be an integer-valued function on $V(G)$.
If  for all $S\subseteq V(G)$, 
$$\OMEGA_l(G\setminus S)\le \sum_{v\in S}\big(h(v)-l(v)\big)+l(G)-e_G(S),$$
then  $G$ has an  $l$-partition-connected  spanning subgraph  $H$ such that for each $v\in X$, 
$d_H(v)\le h(v).$
\end{thm}
Before starting the proof,  let us to state the following two lemmas.
\begin{lem}{\rm (Frank~\cite{MR579073})}\label{lem:Frank}
{Let $G$ be a   graph and let $\ell$ 
be an intersecting  supermodular nonnegative integer-valued   function on
%*
 subsets of $V(G)$ with $\ell(\emptyset)=\ell(G)=0$.
Then $G$ is $\ell$-partition-connected if and only if it has an $\ell$-arc-connected orientation.
}\end{lem}
Note that one can apply Theorem 2 in~\cite{MR642039} instead  of the above-mentioned lemma to obtain further improvement.
Hence we state the following lemma in a  more general version. This can also be extended to a hypergraph version in the same way, 
which   along with  Theorem 3.2 in~\cite{MR2021108} can provide an alternative proof for a special case of Theorem~\ref{thm:hypergraph:toughness:degrees}
\begin{lem}\label{lem:minimal-directed}
{Let $G$ be a directed   
graph  and let $\ell$ be an element-nonincreasing  positively  intersecting supermodular  nonnegative integer-valued  function on 
 subsets of $V(G)$
with  $\ell(\emptyset)=\ell(G)=0$.
If  $H$ is  a minimally $\ell$-arc-connected spanning subdigraph  of $G$,
 then  for each vertex $v$, we must have $d^-_H(v)=\ell(v)$. 
}\end{lem}
\begin{proof}
{Suppose,  by way of contradiction, that $0 \le \ell(u) < d^-_H(u)$ for a vertex $u$.
Let $e=vu$ be a directed edge.
Since $|E(H)|$ is minimal, 
there is a vertex set $A$ including $u$ excluding $v$
such that $\ell(A)=d^-_H(A)> 0$; 
otherwise, the edge  $vu$ can be deleted  from $H$.
Consider $A$ with  minimal $|A|$.
Since $\ell$ is element-nonincreasing,
$d_H^-(u)>\ell(u)\ge  \ell(A)= d_H^-(A).$
Thus there is a directed edge $wu$ with $w\in A$.
Corresponding to $uw$, there is again a vertex set $B$ including $u$  excluding $w$
such that $\ell(B)=d^-_H(B)> 0$.
Therefore, 
$$\ell(A)+\ell(B)=d^-_H(A)+d^-_H(B) \ge d^-_H(A\cap B)+d^-_H(A\cup B) \ge \ell (A\cap B)+\ell (A\cup B).$$
Since $\ell$ is positively  intersecting supermodular and  $u\in A\cap B$, we must have $d^-_H(A\cap B)=\ell(A\cap B)$. 
Since  $A\cap B$ includes $u$ and  $|A\cap B| <|A|$,  we arrive at a contradiction.
}\end{proof}
Now, we are ready   to state the second proof of the above-mentioned theorem.
\begin{proofs}{
\noindent
\hspace*{-5mm}
\textbf{  Theorem~\ref{thm:alternative}.}
Let $r_0$ be a fixed vertex. 
For each vertex $v$, define
$$\ell'(v)=\max\{\ell(v), d_G(v)-h(v)+\ell(v)\},$$ where
$\ell(v)=l(v)-l(G)$ when $v=r_0$, and $\ell(v)=l(v)$ when $v\neq r_0$.
For all vertex  sets  $A$ including $r_0$ with $A\ge 2$,
define $\ell'(A)=\ell(A)=l(A) -l(G)$,
and for all vertex  sets $A$ excluding $r_0$ with $A\ge 2$, 
define $\ell'(A)=\ell(A)=l(A)$.
Let $P$ be a partition of $V(G)$ and 
take $S$  to be the set of all vertices $v$ with $\{ v\} \in P$ such that $\ell'(v)=d_G(v)-h(v)+\ell(v)$. 
Also, define $\mathcal{P}$ to  be the set of all vertex sets $A \in P$ such that $A\neq \{v\}$, when $v\in S$.
Note that for every $A\in \mathcal{P}$, $\ell'(A)=\ell(A)$.
It is not hard  to check that  $\ell$  is an  intersecting supermodular  nonnegative integer-valued set function with $\ell(G)=0$, and so does $\ell'$.
According to the assumption, 
$$\sum_{A\in \mathcal{P}} l(A) -e_{G\setminus S} (\mathcal{P}) \le 
\OMEGA_l(G\setminus S) \le 
 \sum_{v\in S}\big(h(v)-l(v)\big)+l(G)-e_G(S).$$
Since $e_G(P)=\sum_{v\in S}d_G(v)-e_G(S)+e_{G\setminus S} (\mathcal{P})$, we must have 
$$e_G(P) \ge
 \sum_{A\in \mathcal{P}} l(A)  +\sum_{v\in S} (d_G(v)-h(v)+l(v)) -l(G)=\sum_{A\in \mathcal{P}}\ell'(A)+  \sum_{v\in S}\ell'(v)=
\sum_{A\in P}\ell'(A).$$
Thus $G$ is $\ell'$-partition-connected.
By Lemma~\ref{lem:Frank},
the graph $G$ has an $\ell'$-arc-connected orientation so that for every  vertex set $A$, $d^-_G(A)\ge \ell'(A)\ge  \ell(A)$.
In particular, 
for each vertex $v$, 
$d^-_G(v)\ge \ell'(v)\ge d_G(v)-h(v)+\ell(v)$ 
which implies that $d^+_G(v)\le  h(v)-\ell(v)$.
Let $H$ be a minimally $\ell$-arc-connected spanning subdigraph  of $G$.
By Lemma~\ref{lem:minimal-directed}, for each vertex $v$, $d^-_H(v)=\ell(v)$, and so
$$d_H(v)=
d^-_H(v)+d^+_H(v)
 \le  \ell(v)+d^+_G(v) \le  
\ell(v)+
(h(v)-\ell(v))=
h(v).$$
For every partition $P$ of $V(H)$, we  have
$$e_H(P) \ge \sum_{A\in P} d^-_H(A)\ge\sum_{A\in P } \ell(A) = \sum_{A\in P} l(A) -l(G).$$
Hence $H$ is also $l$-partition-connected and the proof is completed.
}\end{proofs}
%
%
%
%==============================================================================
%
%
%
%
%
%
%
%
%
%
%
%
%
%==============================================================================
%
%
\subsection{Graphs with high edge-connectivity}
\label{subsec:edge-connected}
Highly edge-connected graphs are natural candidates for graphs satisfying the assumptions of Theorem~\ref{thm:sufficient}.
We  examine them in this subsection, beginning with
 the following lemma.
\begin{lem}\label{lem:high-edge-connectivity:Theta}
{Let $G$ be a graph,  let $l$ be an intersecting supermodular real function on
%* 
subsets of $V(G)$,
and let $k$ be a positive real number. If $ S\subseteq V(G)$, then
$$\OMEGA_l(G\setminus S)\le
 \begin{cases}
\sum_{v\in S}\frac{d_G(v)}{k}\,-\frac{2}{k}e_G(S),	&\text{when $G$ is $kl$-edge-connected, $k\ge 2$, and  $S\neq \emptyset$};\\ 
 \sum_{v\in S}\big(\frac{d_G(v)}{k}-l(v)\big)+l(G)-\frac{1}{k}e_G(S),	&\text{when $G$ is $kl$-partition-connected and $k\ge 1$}.
\end {cases}$$
}\end{lem}
\begin{proof}
{Let $P$ be the  partition of $V(G)\setminus S$ 
obtained from the $l$-partition-connected components of
 $G\setminus S$.  Obviously, we have
$$e_G( P\cup \{\{v\}:v\in S\})=\sum_{v\in S}d_G(v)\,-e_G(S)+e_{G\setminus S}(P).$$
If $G$ is $kl$-edge-connected and $S\neq \emptyset$,
 there are at least $kl(A)$ edges of $G$ with exactly one end in $A$, for any
$A\in P$. 
Thus
$e_G( P\cup \{\{v\}:v\in S\})\ge \sum_{A\in P}k l(A)-e_{G\setminus S}(P)+e_G(S)$
and so  if $k\ge 2$,  then 
$$k\OMEGA_l(G\setminus S)= \sum_{A\in P}kl(A)-ke_{G\setminus S}(P)\le
 \sum_{A\in P}kl(A)-2e_{G\setminus S}(P)\le
 \sum_{v\in S}d_G(v)\,-2e_G(S).$$
When  $G$ is $kl$-partition-connected, we have 
$e_G( P\cup \{\{v\}:v\in S\})\ge \sum_{A\in P}kl(A)+\sum_{v\in S} kl(v)-kl(G)$
 and so if $k\ge 1$, then
$$k\OMEGA_l(G\setminus S)=  \sum_{A\in P}kl(A)-ke_{G\setminus S}(P)\le 
 \sum_{A\in P}kl(A)-e_{G\setminus S}(P)\le 
\sum_{v\in S}\big(d_G(v)-kl(v)\big)+kl(G)-e_G(S).$$
These inequalities complete the proof.
}\end{proof}
Now, we are ready   to generalize a result in~\cite{II} as the following theorem.
\begin{thm}\label{thm:main:result}
{Let $G$ be a graph with $X\subseteq V(G)$,  let $l$ be an intersecting supermodular element-subadditive nonnegative integer-valued function on  
%*
subsets of $V(G)$,
and let $k$ be a positive real number. Then $G$  has an $l$-partition-connected spanning subgraph $H$ 
such that for each~$v\in X$, 
$$d_H(v)\le
 \begin{cases}
 \big\lceil \frac{1}{k}(d_G(v)-2l(v))\big \rceil+2l(v),	&\text{if $G$ is $kl$-edge-connected and $k\ge 2$};\\ 
 \big\lceil \frac{1}{k}(d_G(v)-l(v))\big\rceil+l(v),	&\text{if $G$ is $kl$-partition-connected and $k\ge 1$};\\
 \big\lceil \frac{1}{k}d_G(v)\big \rceil+l(v),	&\text{if $G$ is $kl$-edge-connected, $k\ge 2$,  and  $X$ is independent};\\ 
 \big\lceil \frac{1}{k}d_G(v)\big\rceil,	&\text{if $G$ is $kl$-partition-connected, $k\ge 1$, and $X$ is independent}.
\end {cases}$$ 
}\end{thm}
\begin{proof}
{Let $S\subseteq V(G)$.
If $G$ is $kl$-edge-connected,  $k\ge 2$, and $S\neq \emptyset$, then
by Lemma~\ref{lem:high-edge-connectivity:Theta}, we have 
$$\OMEGA_l(G\setminus S)\le \sum_{v\in S}\frac{d_G(v)}{k}\,-\frac{2}{k}e_G(S)\le 
 \sum_{v\in S}(\eta(v)-2l(v))+l(G)+l(S)-\frac{2}{k}(e_G(S)+l(S)),$$
where 
for each vertex $v$,
 $\eta(v)= \frac{d_G(v)}{k}+2l(v)$.
Note that when $G$ is $kl$-edge-connected,  $k\ge 2$, and $S=\emptyset$, we must have $\OMEGA_l(G\setminus S) = l(G)$.
If $G$ is $kl$-partition-connected and $k\ge 1$, then  by
 Lemma~\ref{lem:high-edge-connectivity:Theta}, we also have 
$$\OMEGA_l(G\setminus S)\le 
 \sum_{v\in S}\big(\frac{d_G(v)}{k}-l(v)\big)+l(G)-\frac{1}{k}e_G(S)
\le 
 \sum_{v\in S}(\eta(v)-2l(v))+l(G)+l(S)-\frac{1}{k}(e_G(S)+l(S)),$$
where 
for each vertex $v$,
 $\eta(v)= \frac{d_G(v)}{k}+l(v)$.
Thus  the first two assertions follow from Theorem~\ref{thm:sufficient} for  $\lambda\in \{2/k, 1/k\}$.
The second two assertions can similarly be proved.
}\end{proof}
%
%
%
%
%\subsection{Graphs with high minimum degree}
%
The following corollary can improve a result in~\cite{MR3192386} by replacing minimum degree condition.
We denote below by $\delta^+(G)$ the minimum out-degree of a directed graph $G$.
\begin{cor}
{Let $G$ be a  graph with  $X\subseteq V(G)$ and let $k$ be a real number with $k\ge 1$. If $G$ has an orientation with
$\delta^+(G)\ge km$, then it has a spanning subgraph $H$ with a new orientation such that
$\delta^+(H)\ge m$ and 
 for each 
$v\in X$, 
$$d_H(v)\le
 \begin{cases}

 \big\lceil \frac{1}{k}d_G(v)\big\rceil,	&\text{when  $X$ is independent};\\
 \big\lceil \frac{1}{k}(d_G(v)-m)\big\rceil+m,	&\text{otherwise}.
\end {cases}$$
}\end{cor}
\begin{proof}
{Since $\delta^+(G)\ge km$, the graph $G$ is $kl$-partition-connected, where $l(v)=m$ for each vertex $v$ and $l(A)=0$ for every vertex set $A$ with $|A|\ge 2$.
Let $H$ be an $l$-partition-connected spanning subgraph of $G$ with the properties described in Theorem~\ref{thm:main:result}.
Since $H$ is $l$-partition-connected, by Lemma~\ref{lem:Frank},  it has an orientation such that $\delta^+(H)\ge m$.
}\end{proof}
%
%
%
%=============================================================================
%
%
%
%
%
%
%=============================================================================
%
%
\section{Highly partition-connected spanning subgraphs with bounded  degrees}
\label{sec:Generalizations}
In this section, we shall  strengthen  Theorem~\ref{thm:sufficient} for finding partition-connected spanning graphs with bounded degrees, when $l$ is nonincreasing.
Before doing so, 
we  establish  the following promised generalization of
Theorems~\ref{thm:preliminary:structure}. Note that $\OMEGA(G\setminus [S,F])=\OMEGA(G\setminus S)+\sum_{v\in S}l(v)$, when $F$ is the trivial spanning subgraph and $l$ is element-subadditive.
\begin{thm}\label{thm:preliminary:generalization}
{Let $G$ be an $l$-partition-connected  graph with the  spanning $l$-sparse subgraph  $F$, where $l$
 is a   intersecting supermodular weakly subadditive   integer-valued function on 
%*
 subsets of $V(G)$.
Let  $h$ be an  integer-valued function on $V(G)$.
 If $H$ is  a minimally  $l$-partition-connected spanning subgraph of $G$ containing $F$  with the minimum total excess from $h+d_F$,  then there exists a subset $S$ of  $V(G)$ with the following properties:
\begin{enumerate}{
\item $\OMEGA_l(G\setminus [S,F])=\OMEGA_l(H\setminus [S,F])$.\label{Condition F2.1}
\item $S\supseteq \{v\in V(G):d_H(v)> h(v)+d_F(v)\}$.\label{Condition F2.2}
\item For each vertex $v$ of $S$, $d_H(v)\ge h(v)+d_F(v)$.\label{Condition F2.3}
}\end{enumerate}
}\end{thm}
\begin{proof}
{Define $V_0=\emptyset $ and  $V_1=\{v\in V(H): d_H(v)> h(v)+d_F(v)\}$.
For any $S\subseteq V(G)$ and   $u\in V(G)\setminus S$, 
let $\mathcal{A}(S, u)$ be the set of all  minimally $l$-partition-connected spanning subgraphs 
$H^\prime $ of $G$ containing $F$
such that    $d_{H'}(v)\le h(v)+d_F(v)$ for all $v\in V(G)\setminus V_1$,  
 and $H^\prime$ and $H$ have the same edges, except for some of the edges of $G$ whose ends are in $X$, 
where $H[X]$ is the $l$-partition-connected  component of $H\setminus [S,F]$ containing $u$.
Note that $H'[X]$ must  automatically be $l$-partition-connected.
Now, for each integer $n$ with $n\ge 2$,   recursively define $V_n$ as follows:
$$V_n=V_{n-1} \cup \{\, v\in V(G)\setminus V_{n-1} \colon \, d_{H^\prime }(v)= h(v)+d_F(v) \text{,\, for all\, }H^\prime \in \mathcal{A}(V_{n-1},v)\,\}.$$
 Now, we prove the following claim.
%
%===================================	Claim
\vspace{2mm}\\
{\bf Claim.} 
Let $x$ and $y$ be two vertices in different $l$-partition-connected components of $H\setminus [V_{n-1},F]$.
If $xy\in E(G)\setminus E(H)$, then $x\in V_{n}$ or $y\in V_{n}$.
\vspace{2mm}\\
%===================================
{\bf Proof of Claim.} 
By induction on $n$.
 For $n=1$, the proof is clear. 
Assume that the claim is true for $n-1$.
Now we prove it for $n$.
Suppose otherwise that vertices $x$ and $y$ are in different  $l$-partition-connected components of $H\setminus [V_{n-1},F]$, respectively, with the vertex sets $X$ and $Y$,
$xy\in E(G)\setminus E(H)$, 
and $x,y\not \in V_{n}$. 
Since $x,y\not\in  V_{n}$,
 there exist
 $H_x\in \mathcal{A}(V_{n-1},x)$ and
 $H_y\in \mathcal{A}(V_{n-1},y)$ with $d_{H_x}(x)< h(x)+d_F(x)$ and $d_{H_y}(y)< h(y)+d_F(y)$. 
By the induction hypothesis,
 $x$ and $y$ are in the same $l$-partition-connected component 
% Z
of $H\setminus [V_{n-2},F]$ with the vertex set $Z$ so that $X\cup Y \subseteq Z$.
Let $M$  be  the nonempty set of  edges of $H[Z]\setminus E(F)$ incident to the vertices in $V_{n-1}\setminus V_{n-2}$
 whose ends lie in different $l$-partition-connected components of $H[Z]\setminus [Z\cap V_{n-1}, F]$.
By Proposition~\ref{prop:replacing},
 there  exists an edge $zz'\in M$ with $z\in Z \cap V_{n-1}$  such that 
$H[Z]-zz'+xy$ is $l$-partition-connected.
Now, let  $H'$ be the spanning subgraph of $G$   containing    $F$ with
 $$E(H')=E(H)-zz'+xy
-E(H[X])+E(H_x[X])
-E(H[Y])+E(H_y[Y]).$$
By  repeatedly  applying Proposition~\ref{prop:deducing},  one can easily check that $H'$  
 is  $l$-partition-connected.
For each  $v\in V(H')$, we  have
$$d_{H'}(v)\le   
 \begin{cases}
d_{H_v}(v)+1,	&\text{if   $v\in \{x,y\}$};\\
d_{H}(v),	&\text{if   $v=z'$},
\end {cases}
\quad \text{ and }\quad
 d_{H'}(v)= 
 \begin{cases}
d_{H_x}(v),	&\text{if   $v\in X\setminus \{x,z'\}$};\\
d_{H_y}(v),	&\text{if   $v\in Y\setminus \{y,z'\}$};\\
d_{H}(v),	&\text{if   $v\notin X\cup Y\cup \{z,z'\}$}.
\end {cases}$$
If $n\ge 3$, then it is not hard to see that  $d_{H'}(z)<d_{H}(z)\le h(z)+d_F(z)$ and $H'$ lies in $ \mathcal{A}(V_{n-2}, z)$.
Since $z\in V_{n-1}\setminus V_{n-2}$,  we arrive at a contradiction.
For the case $n=2$,  
since $z\in V_1$, it is easy to see that
$h(z)+d_F(z)\le d_{H'}(z)<d_{H}(z)$ and  $te(H',h+d_F)< te(H,h+d_F)$,   which is again a contradiction.
Hence the claim holds.
%
%=====================================  End Claim
%
%

Obviously, there exists a positive integer $n$ such that  $V_1\subseteq \cdots\subseteq V_{n-1}=V_{n}$.
 Put $S=V_{n}$.  
Since $S\supseteq V_1$, Condition~\ref{Condition F2.2} clearly holds.
For each $v\in V_i\setminus V_{i-1}$ with $i\ge 2$, 
we have $H\in \mathcal{A}(V_{i-1},v)$ and so  $d_H(v)\ge h(v)+d_F(v)$. 
This establishes   Condition~\ref{Condition  F2.3}.
 Because $S=V_{n}$, 
the previous claim implies Condition~\ref{Condition F2.1} and  completes the proof.
}\end{proof}
In the above-mentioned theorem, 
we could assume that $\OMEGA_l(H)=\OMEGA_l(G)$ and   choose $H$ with the minimum $te (H,h+d_F)$,  
whether $G$ is $l$-partition-connected or not.
Conversely, if we  assume  that $te(H,h+d_F)=0$ and  choose $H$ with the minimum $\OMEGA_l(H)$,
 the next theorem can be derived.
However, the above-mentioned theorem works remarkably well, we shall use  this result  
to  get further improvement  in the last subsection.
\begin{thm}\label{thm:final:minor-improvement}
{Let $G$ be a  graph with the  spanning $l$-sparse subgraph  $F$, where $l$
 is an intersecting supermodular weakly subadditive integer-valued  function on 
%*
 subsets of $V(G)$.
Let  $h$ be an  integer-valued function on $V(G)$.
 If $H$ is  a  spanning subgraph of $G$ containing $F$ with $te(H,h+d_F)=0$ and with the minimum $\OMEGA_l(H)$,  
then there exists a subset $S$ of  $V(G)$ with the following properties:
\begin{enumerate}{
\item $\OMEGA_l(G\setminus [S,F])=\OMEGA_l(H\setminus [S,F])$.\label{condition:final:1}
\item For each vertex $v$ of $S$, $d_H(v)= h(v)+d_F(v)$.\label{condition:final:2}
}\end{enumerate}
}\end{thm}
\begin{proof}
{Define $V_0=\emptyset $. 
For any $S\subseteq V(G)$ and   $u\in V(G)\setminus S$, 
let $\mathcal{A}(S, u)$ be the set of all  spanning subgraphs 
$H^\prime $ of $G$ containing $F$
with  $te(H',h+d_F)=0$ such that   $\OMEGA_l(H')=\OMEGA_l(H)$, $H'[X]$ is  $l$-partition-connected,
$H^\prime$ and $H$ have the same edges, except for some of the edges of $G$ whose ends are in $X$, 
where $H[X]$ is the $l$-partition-connected  component of $H\setminus [S,F]$ containing $u$.
Now, for each integer $n$ with $n\ge 2$,   recursively define $V_n$ as follows:
$$V_n=V_{n-1} \cup \{\, v\in V(G)\setminus V_{n-1} \colon \, d_{H^\prime }(v)= h(v)+d_F(v) \text{,\, for all\, }H^\prime \in \mathcal{A}(V_{n-1},v)\,\}.$$
 Now, we prove the following claim.
%
%===================================	Claim
\vspace{2mm}\\
{\bf Claim.} 
Let $x$ and $y$ be two vertices in different $l$-partition-connected components of $H\setminus [V_{n-1},F]$.
If $xy\in E(G)\setminus E(H)$, then $x\in V_{n}$ or $y\in V_{n}$.
\vspace{2mm}\\
%===================================
{\bf Proof of Claim.} 
By induction on $n$.
Suppose otherwise that vertices $x$ and $y$ are in different  $l$-partition-connected components of $H\setminus [V_{n-1},F]$, respectively, with the vertex sets $X$ and $Y$,
$xy\in E(G)\setminus E(H)$, 
and $x,y\not \in V_{n}$. 
Since $x,y\not\in  V_{n}$,
 there exist
 $H_x\in \mathcal{A}(V_{n-1},x)$ and
 $H_y\in \mathcal{A}(V_{n-1},y)$ with $d_{H_x}(x)< h(x)+d_F(x)$ and $d_{H_y}(y)< h(y)+d_F(y)$. 
 For $n=1$, define  $H'$ to be the spanning subgraph of $G$ containing $F$ with
 $$E(H')=E(H)+xy
-E(H[X])+E(H_x[X])
-E(H[Y])+E(H_y[Y]).$$
Since  the edge $xy$ joins different $l$-partition-connected components of $H$, we must have
$\OMEGA_l(H') < \OMEGA_l(H)$.
Since   $te (H',h+d_F)=0$, we arrive at a contradiction.
Now, suppose $n\ge 2$.
By the induction hypothesis,
 $x$ and $y$ are in the same $l$-partition-connected component 
% Z
of $H\setminus [V_{n-2},F]$ with the vertex set $Z$ so that $X\cup Y \subseteq Z$.
Let $M$  be  the nonempty set of  edges of $H[Z]\setminus E(F)$ incident to the vertices in $V_{n-1}\setminus V_{n-2}$
 whose ends lie in different $l$-partition-connected components of $H[Z]\setminus [Z\cap V_{n-1}, F]$.
By Proposition~\ref{prop:replacing},
 there  exists an edge $zz'\in M$ with $z\in Z \cap V_{n-1}$  such that 
$H[Z]-zz'+xy$ is $l$-partition-connected.
Now, let  $H'$ be the spanning subgraph of $G$   containing    $F$ with
 $$E(H')=E(H)-zz'+xy
-E(H[X])+E(H_x[X])
-E(H[Y])+E(H_y[Y]).$$
It  is easy  to see that the  $l$-partition-connected components of $H'$ and $H$ have the same vertex sets.  
Since  $H$ and $H'$ have  the same edges  joining these $l$-partition-connected components, $\OMEGA_l(H') = \OMEGA_l(H)$.
For each  $v\in V(H')$, we  have
$$d_{H'}(v)\le   
 \begin{cases}
d_{H_v}(v)+1,	&\text{if   $v\in \{x,y\}$};\\
d_{H}(v),	&\text{if   $v=z'$},
\end {cases}
\quad \text{ and }\quad
 d_{H'}(v)= 
 \begin{cases}
d_{H_x}(v),	&\text{if   $v\in X\setminus \{x,z'\}$};\\
d_{H_y}(v),	&\text{if   $v\in Y\setminus \{y,z'\}$};\\
d_{H}(v),	&\text{if   $v\notin X\cup Y\cup \{z,z'\}$}.
\end {cases}$$
It is not hard to check that  $d_{H'}(z)<d_{H}(z)\le h(z)+d_F(z)$ and $H'$ lies in $ \mathcal{A}(V_{n-2}, z)$.
Since $z\in V_{n-1}\setminus V_{n-2}$,  we arrive at a contradiction.
Hence the claim holds.
%
%=====================================  End Claim
%

%
Obviously, there exists a positive integer $n$ such that  $V_1\subseteq \cdots\subseteq V_{n-1}=V_{n}$.
 Put $S=V_{n}$.  
For each $v\in V_i\setminus V_{i-1}$, 
we have $H\in \mathcal{A}(V_{i-1},v)$ and so  $d_H(v)= h(v)+d_F(v)$. 
This establishes   Condition~\ref{condition:final:2}.
 Because $S=V_{n}$, 
the previous claim implies Condition~\ref{condition:final:1} and  completes the proof.
}\end{proof}
%==============================================================================
%
%
%
\subsection{Prerequisites}
The following lemma provides  a  generalization for 
Lemma~\ref{lem:minimally-omega}.
Recall that $\OMEGA_l(H\setminus [S,F])=\OMEGA_l(H\setminus S)+\sum_{v\in S} l(v)$ when $F$ is the trivial spanning subgraph and $l$ is element-subadditive.
\begin{lem}\label{lem:spanningforest:gen}
{Let $H$ be an $l$-sparse graph  with the spanning subgraph  $F$, where  $l$ is an intersecting supermodular weakly subadditive integer-valued function on 
%*
 subsets of $V(H)$. 
 If $S\subseteq V(H)$ and  $\mathcal{F}=H\setminus E(F)$, then 
$$\sum_{v\in S}d_{\mathcal{F}}(v)=\OMEGA_l(H\setminus [S,F])-\OMEGA_l(H)+e_\mathcal{F}(S).$$
}\end{lem}
\begin{proof}
{By induction on the number  of edges of $\mathcal{F}$ which are  incident to the vertices in $S$. If there is no edge of   $\mathcal{F}$  incident to a vertex  in $S$, then the proof is clear. Now, suppose that there exists an edge $e=uu'\in E(\mathcal{F})$ with $|S\cap \{u,u'\}|\ge 1$. Hence
\begin{enumerate}{
\item $\OMEGA_l(H)=\OMEGA_l (H\setminus e)-1,$
\item $\OMEGA_l(H\setminus [S,F])= \OMEGA_l((H\setminus e)\setminus [S,F]),$

\item $e_\mathcal{F}(S)=e_{\mathcal{F}\setminus e}(S)+|S\cap \{u,u'\}|-1,$
\item $\sum_{v\in S}d_{\mathcal{F}}(v)=\sum_{v\in S}d_{\mathcal{F}\setminus e}(v)\,+|S\cap \{u,u'\}|.$
}\end{enumerate}
Therefore, by the induction hypothesis on  $H\setminus e$ with the spanning subgraph $F$ the lemma holds.
}\end{proof}
The following lemma   provides a useful relationship between two parameters
$\OMEGA_l(G \setminus S)$
and $\OMEGA_l(G \setminus [S,F])$, when $l$ is nonincreasing.
We shall apply it in the subsequent subsections.
\begin{lem}\label{lem:gen}
{Let $G$ be  a  graph with  the spanning subgraph $F$ and let $l$ be a nonincreasing intersecting supermodular nonnegative  integer-valued  function on
%*
 subsets of $V(G)$. 
If $S\subseteq V(G)$ then
$$  \OMEGA_l(G\setminus [S,F])\le
 \OMEGA_l(G\setminus S)
+\sum_{v\in S}\max\{0,l(v)- d_F(v)\}
+e_F(S).$$
Furthermore,
$  \OMEGA_l(G\setminus [S,F])\le
 \OMEGA_l(G\setminus S)
+\frac{1}{(c-1)}e_{F}(S),$
when every $l$-partition-connected component $C$ of $F$ we have $\sum_{v\in C}{l(v)}\ge cl(C)-\frac{c-1}{2}d_F(C)$ and  $c\ge 2$.
}\end{lem}
\begin{proof}
{Define $P$ and $P'$ to be the partitions of $V(G)$ and $V(G)\setminus S$ obtained from the $l$-partition-connected components of $G\setminus [S,F]$ and $G\setminus S$.
Set $R=\{A\in P: A\subseteq S\}$, $R_1=\{A\in R:  |A|=1\}$, and $R_2=\{A\in R: |A|\ge 2\}$.
It is not difficult to check that
$$ e_{G\setminus [S,F] }(P)\ge
 e_{G\setminus S}(P') -\sum_{A\in P\setminus R}e_{G[A\setminus S]}(P'_{A\setminus S}) + D_F(R),$$
 where
 $P'_{A\setminus S}$ denotes the partition of $A\setminus S$ obtained from vertex sets of $P'$, and 
$D_F(R)$ denotes the number of edges of $F$ joining different parts of $P$ incident to vertex sets in $R$.
Thus
$$\OMEGA_l(G\setminus [S, F])-\sum_{A\in P}l(A)  \le \OMEGA_l(G\setminus S)
-\sum_{A\in P\setminus R} \OMEGA_l(G[A\setminus S])-D_F(R).$$
Since $\OMEGA_l(G[A\setminus S]) \ge l(A\setminus S)$, for any  $A\in P\setminus  R$, we have
$$\OMEGA_l(G\setminus [S, F]) \le 
\OMEGA_l(G\setminus S)
+\sum_{A\in P\setminus R} {(l(A)-l(A\setminus S))}
+\sum_{A\in R} {l(A)}
-D_F(R).$$
Since $l$ is nonincreasing,
$$\OMEGA_l(G\setminus [S, F]) \le 
\OMEGA_l(G\setminus S)
+\sum_{A\in R} {l(A)}-D_F(R).$$
In  the first statement,  $e_F(A)\ge \sum_{v\in A}l(v)-l(A)\ge l(A)$, for any $A\in R_2$, and so
$$\sum_{A\in R}l(A)-D_F(R)  \le  
\sum_{\{v\}\in R_1}l(v)+\sum_{A\in R_2}e_F(A)-\sum_{\{v\}\in R_1} d_F(v)+e_F(R_1)
\le  \sum_{\{v\}\in R_1} (l(v)-d_F(v))+e_F(S).$$
 Therefore,
$$\OMEGA_l(G\setminus [S, F]) 
\le    \OMEGA_l(G\setminus S) +
e_F(S)+\sum_{v\in S} \max\{0,l(v)- d_F(v)\}.$$
In  the second statement, $\sum_{v\in A}{l(v)}\ge cl(A)-\frac{c-1}{2}d_F(A)$ for any $A\in R$, and so
$$ \sum_{A\in R}\big(l(A)-\frac{1}{c-1}(\sum_{v\in A}l(v)-l(A))\big)\le \sum_{A\in R}\frac{1}{2} d_F(A) \le D_F(R).$$
Since  $e_F(A)\ge \sum_{v\in A}l(v)-l(A)$, for any $A\in R$, it is easy to check that
$$ \OMEGA_l(G\setminus [S,F])\le
 \OMEGA_l(G\setminus S)
+\frac{1}{(c-1)}e_{F}(S).$$
Hence the lemma holds.
}\end{proof}
\subsection{A strengthened version of a special case of  Theorem~\ref{thm:sufficient}}
A strengthened version   of Theorems~\ref{thm:sufficient}
is given in the following theorem, when $l$ is nonincreasing.
\begin{thm}\label{thm:Second-Generalization}
Let $G$ be a graph with $X\subseteq V(G)$  and with the spanning subgraph $F$, and
let $l$ be a nonincreasing intersecting supermodular  nonnegative integer-valued function on 
%*
subsets of $V(G)$.
 Let  $\lambda \in [0,1]$ be a real number and  let  $\eta$ be a  real function on $X$.
If  for all $S\subseteq X$, 
$$\OMEGA_l(G\setminus S)< 1+\sum_{v\in S}\big(\eta(v)-2l(v)\big)+l(G)+l(S)-\lambda(e^*_G(S)+l(S)),$$
then  it  has an $l$-partition-connected  spanning subgraph  $H$ containing $F$ such that for each $v\in X$,
 $$d_H(v)\le \big\lceil \eta(v)-\lambda l(v)\big\rceil+\max\{0,d_F(v)-l(v)\}.$$
\end{thm}
\begin{proof}
{For each vertex $v$, define 
$$h(v)= 
 \begin{cases}
d_G(v)+1,	&\text{if   $v\not\in X$};\\
\lceil \eta(v)-\lambda l(v)\rceil -\min \{l(v), d_F(v)\},	&\text{if $v\in X$}.
\end {cases}$$
First, suppose that $F$ is $l$-sparse.
Note that $G$ is automatically $l$-partition-connected, because  of $\OMEGA_l(G\setminus \emptyset) \le l(G)$.
Let $H$ be a minimally $l$-partition-connected   spanning subgraph  of $G$ containing $F$ with the minimum total excess from 
$h+d_F$.
Define $S$ to be a subset of $V (G)$  with the properties described in 
Theorem~\ref{thm:preliminary:generalization}.
Obviously, $S\subseteq X$. Put $\mathcal{F}=H\setminus E(F)$.
By Lemma~\ref{lem:spanningforest:gen},
$$\sum_{v\in S} h(v)  +te(H, h+d_F)= \sum_{v\in S} d_\mathcal{F}(v) = 
\OMEGA_l(H\setminus [S, F])  -l(G)+ e_{\mathcal{F}}(S),$$
and so
$$\sum_{v\in S} h(v)  +te(H, h+d_F) =  \OMEGA_l(G\setminus [S, F])-l(G)  + e_{\mathcal{F}}(S).$$
Since $e_F(S)+e_\mathcal{F}(S)=e_H(S)$,  Lemma~\ref{lem:gen} implies that
\begin{equation}\label{thm:Gen:eq:1}
\sum_{v\in S} (h(v)- \max\{0, l(v)-d_F(v)\})  +te(H, h+d_F)
\le  \OMEGA_l(G\setminus S)  -l(G) +e_H(S).
\end{equation}
Also, by the assumption,
\begin{equation}\label{thm:Gen:eq:2}
\OMEGA_l(G\setminus S) -l(G)+ e_{H}(S) < 1+\sum_{v\in S}\big(\eta(v)-2l(v)\big)-\lambda(e_G(S)+l(S))+e_H(S)+l(S).
\end{equation}
Since $e_H(S)\le e^*_G(S)$ and $e_H(S)\le \sum_{v\in S} l(v)-l(S)$,
\begin{equation}\label{thm:Gen:eq:3}
-\lambda(e^*_G(S)+l(S))+e_H(S)+l(S) \le -\lambda(e_H(S)+l(S))+e_H(S)+l(S)\le (1-\lambda)\sum_{v\in S}{l(v)}.
\end{equation}
Therefore, Relations (\ref{thm:Gen:eq:1}),  (\ref{thm:Gen:eq:2}), and (\ref{thm:Gen:eq:3})  can conclude that
$$ \sum_{v\in S} (h(v)-\max\{0,l(v)-d_F(v)\})\,+te(H, h+d_F) 
< 1 +\sum_{v\in S} (\eta (v)-\lambda l(v)-l(v)).$$
On the other hand, by the definition of $h(v)$,
$$\sum_{v\in S} \big(\eta (v)-\lambda l(v)-l(v)-h(v)+\max\{0,l(v)-d_F(v)\}\big)\le 0.$$
Hence $te(H,h + d_F) = 0$ and the theorem holds.
Now, suppose that $F$ is not $l$-sparse.
Remove some of the edges of the $l$-partition-connected components  of $F$ until the resulting $l$-sparse  graph $F'$ have the same $l$-partition-connected components. 
For each vertex $v$ with $d_{F'}(v) < d_F(v)$,  
we have $d_F(v) \ge d_{F'}(v) \ge l(v)$, since $v$  must lie in a non-trivial 
$l$-partition-connected component of $F'$ and $l$ is nonincreasing.
It is enough, now,  to apply the theorem on $F'$ and finally add  the edges of $E(F)\setminus E(F')$ 
to that explored $l$-partition-connected spanning subgraph.
}\end{proof}
\subsection{Tough enough graphs}
In this subsection, we improve below  Theorems~\ref{thm:sufficient} for graphs that 
the values of $\OMEGA_l(G\setminus S) $ are small enough compared to $|S|$,  which enables us to choose $\eta(v)$ small enough, in compensation we require that the
given spanning subgraph F approximately have large $l$-partition-connected components.
\begin{thm}\label{thm:gen:tough-enough}
{Let $G$ be a  graph and  $l$ be  a nonincreasing intersecting supermodular 
nonnegative integer-valued function on 
%*
subsets of $V(G)$.
Let $h$ be an  integer-valued function on $V(G)$.
 Let $F$ be a spanning subgraph of $G$ in which  for every $l$-partition-connected component  $C$ of $F$, we have
$\sum_{v\in V(C)}{l(v)}\ge cl(C)-\frac{c-1}{2}d_F(C)$ and  $c\ge 2$.
If for all $S\subseteq V(G)$,
$$ \OMEGA_l(G\setminus S) < 1 +\sum_{v\in S}\big(\frac{c}{2c-2}h(v)-\frac{1}{c-1}l(v)\big)+l(G)+\frac{1}{c-1}l(S),$$
then $G$ has an $l$-partition-connected  spanning subgraph  $H$ containing $F$ 
such that for each vertex $v$, 
$d_H(v)\le h(v)+d_F(v).$
}\end{thm}
\begin{proof}
{First, suppose that $F$ is $l$-sparse.
Note that $G$ is automatically $l$-partition-connected, because  of $\OMEGA_l(G\setminus \emptyset) \le l(G)$.
Let $H$ be an $l$-sparse spanning subgraph  of $G$ containing $F$ with $te(H,h+d_F)=0$ and with the minimum $\OMEGA_l(H)$.
Define $S$ to be a subset of $V (G)$  with the properties described in Theorem~\ref{thm:final:minor-improvement}.
Put $\mathcal{F}=H\setminus E(F)$.
By Lemma~\ref{lem:spanningforest:gen},
$$\sum_{v\in S} h(v) = \sum_{v\in S} d_\mathcal{F}(v) = 
 \OMEGA_l(H\setminus [S, F])  -\OMEGA_l(H)+ e_{\mathcal{F}}(S),$$
and so
\begin{equation}\label{eq:thm:2-connected:0}
 \OMEGA_l(H) =  \OMEGA_l(G\setminus [S, F]) + e_{\mathcal{F}}(S)- \sum_{v\in S} h(v) .
\end{equation}
Since 
$e_\mathcal{F}(S)+e_F(S)=e_{H}(S) \le \sum_{v\in S}l(v)-l(S)$
 and
 $e_\mathcal{F}(S)
\;\le\; \frac{1}{2}\sum_{v\in S}d_{\mathcal{F}}(v) =\frac{1}{2}\sum_{v\in S}h(v)$, we have
\begin{equation}\label{eq:thm:2-connected:1}
e_\mathcal{F}(S)+\frac{1}{c-1}e_F(S)
\le 
\frac{1}{2}\sum_{v\in S}h(v)+\frac{1}{c-1}\big( \sum_{v\in S}l(v)-l(S)-\frac{1}{2}\sum_{v\in S}h(v)\big).
\end{equation}
Also,  by Lemma~\ref{lem:gen},
\begin{equation}\label{eq:thm:2-connected:2}
\OMEGA_l(G\setminus [S,F]) + e_{\mathcal{F}}(S)\le 
\OMEGA_l(G\setminus S)+e_{\mathcal{F}}(S)+\frac{1}{c-1}e_F(S).
\end{equation}
Therefore, Relations~(\ref{eq:thm:2-connected:0}),~(\ref{eq:thm:2-connected:1}), and~(\ref{eq:thm:2-connected:2}) can conclude that
\begin{equation*}
\OMEGA_l(H) \le
\OMEGA_l(G\setminus S)-
 \frac{c}{2c-2} 
\sum_{v\in S}h(v)+\frac{ \sum_{v\in S}l(v)-l(S)}{c-1} 
<  l(G)+1.
\end{equation*}
Hence  $\OMEGA_l(H) =l(H)$ and the theorem holds.
Now, suppose that $F$ is not $l$-sparse.
Remove some of the edges of the $l$-partition-connected components  of $F$ until the resulting $l$-sparse  graph $F'$ have the same $l$-partition-connected components. 
For every $l$-partition-connected component $C$ of $F'$, we  still have $d_{F'}(C)=d_F(C)$.
It is enough, now,  to apply the theorem on $F'$ and finally add  the edges of $E(F)\setminus E(F')$ 
to that explored $l$-partition-connected spanning subgraph. 
}\end{proof}
When we consider the special cases $h(v)= 1$, the theorem becomes  simpler  as the following result.
\begin{cor}\label{thm:tough:1-Extend}
{Let $G$ be a  graph and   $l$ be   a nonincreasing intersecting supermodular nonnegative  integer-valued function on
%*
subsets of $V(G)$.
Let $F$ be a spanning subgraph of $G$ in which for  every $l$-partition-connected component  $C$ of $F$, we have
$\sum_{v\in C}{l(v)}\ge cl(C)-\frac{c-1}{2}d_F(C)$ and  $c\ge 2$.
If for all $S\subseteq V(G)$,
$$\OMEGA_l(G\setminus S) \le  
\sum_{v\in S}\frac{c-2ml(v)}{2m(c-1)}+l(G)+\frac{1}{c-1}l(S)$$
then $G$ has an $ml$-partition-connected  spanning subgraph  $H$ containing $F$ 
such that for each vertex $v$, 
 $d_H(v) \le d_F(v)+1$.
}\end{cor}
\begin{proof}
{Let   $G'$  be the union of $m$ copies of $G$ with the same vertex set and define $l'=ml$.
It is easy to check that  $\OMEGA_{l'}(G'\setminus S) =m\OMEGA_l(G\setminus S)$, for every $S\subseteq V(G)$. 
Define  $h(v)=1$ for each vertex $v$.
By  Theorem~\ref{thm:gen:tough-enough}, the graph 
$G'$ has an  $l'$-partition-connected spanning subgraph $H$ containing $F$ such that for each vertex $v$, 
$d_{H}(v)\le  h(v)+d_{F}(v)\le 1+d_F(v)$.
According to the construction,   the graph $H$ must  have no multiple edges of $E(G')\setminus E(F)$.
Hence  $H$ itself is a spanning subgraph of $G$ and the proof is completed.
}\end{proof}
\section{Total excesses from comparable functions}
\label{sec:comparable-functions}
In this section, 
we  formulate  the following  strengthened versions of the main results of this paper which are motivated by  Ozeki-type condition~\cite{MR3386038}.
As their proofs require  only  minor modifications,
we shall only state  the strategy of the proof in the subsequent subsection.
\begin{thm}\label{thm:final:gen}
Let $G$ be an $l$-partition-connected graph, where $l$ is an  intersecting supermodular element-subadditive integer-valued function on
%*
subsets of $V(G)$. 
Let $p$ be a positive integer. 
For each integer $i$ with $1\le i\le p$, 
let $t_i$ be  a nonnegative integer,
 let  $\lambda_i \in [0,1]$ be a real number, and  let 
 $\eta_i$ be a  real function on $V(G)$ with 
$  \eta_1-\lambda_1 l
\ge \cdots
\ge
 \eta_p-\lambda_p l$.
If  for all $S\subseteq V(G)$  and  $i\in \{1,\ldots,p\}$,
$$\OMEGA_l(G\setminus S)<1+ \sum_{v\in S}\big(\eta_i(v)-2l(v)\big)+l(G)+l(S)-\lambda_i(e^*_G(S)+l(S))+t_i,$$
then  $G$ has  an $l$-partition-connected  spanning subgraph  $H$ satisfying
 $te(H,h_i)\le t_i$ for all $i$ with $1\le i\le p$, 
where $h_i(v)=\lceil \eta_i(v)-\lambda_i l(v)\rceil$ for all vertices  $v$.
\end{thm}
\begin{thm}\label{thm:final:gen}
Let $G$ be an $l$-partition-connected graph  
with the spanning subgraph $F$, where $l$ is a nonincreasing  intersecting supermodular  nonnegative integer-valued  function 
on 
%*
 subsets of $V(G)$. 
Let $p$ be a positive integer. 
For each integer $i$ with $1\le i\le p$, 
let $t_i$ be  a nonnegative integer,
 let  $\lambda_i \in [0,1]$ be a real number, and  let 
 $\eta_i$ be a  real function on $V(G)$ with 
$  \eta_1-\lambda_1 l
\ge \cdots
\ge
 \eta_p-\lambda_p l$.
If  for all $S\subseteq V(G)$  and  $i\in \{1,\ldots,p\}$,
$$\OMEGA_l(G\setminus S)<1+ \sum_{v\in S}\big(\eta_i(v)-2l(v)\big)+l(G)+l(S)-\lambda_i(e^*_G(S)+l(S))+t_i,$$
then  $F$ can be extended to  an $l$-partition-connected  spanning subgraph  $H$ satisfying
 $te(H,h_i)\le t_i$ for all $i$ with $1\le i\le p$, 
where $h_i(v)=\lceil \eta_i(v)-\lambda_i l(v)\rceil+\max\{0,d_F(v)-l(v)\}$ for all vertices  $v$.
\end{thm}
\subsection{Strategy of the proof}
Let $G$ be a graph with the  spanning subgraph  $H$ and take $xy\in E(G)\setminus E(H)$. 
Let $h$ be  an integer-valued function  on $V(G)$.
It is easy to check that if   $d_H(x)< h(x)$ and $d_H(y)< h(y)$, then  $te(H+xy,h)=te(H,h)$,
 and also this equality  holds for any other  integer-valued function $h'$  on $V(G)$ with $h'\ge  h$.
This observation was used by Ozeki (2015)  
to prove Theorem~6 in~\cite{MR3386038} with a method that
  decreases total excesses from  comparable functions, step by step, by starting from the largest function to the smallest  function. 
Inspired by Ozeki's method,  we now formulate the following  strengthened  version of Theorem~\ref{thm:preliminary:generalization}.
\begin{thm}\label{thm:preliminary:generalization:excesses}
{Let $G$ be an $l$-partition-connected graph  with the $l$-sparse spanning subgraph $F$, where $l$ is an
intersecting supermodular weakly subadditive integer-valued  function on 
%*
subsets of $V(G)$. 
Let $h_1,\ldots,h_q$ be $q$  integer-valued functions on $V(G)$ with $h_1\ge \cdots \ge h_q$.
Define $\Gamma_0$ to be the set of all  $l$-partition-connected spanning subgraphs  $H$ of $G$ containing $F$.
For each positive integer $n$ with $ n\le q$, recursively define  $\Gamma_n$  
to be the set of all  graphs   $H$  belonging to $\Gamma_{n-1}$   with the smallest $te(H, h_n+d_F)$.
If $H\in \Gamma_{q}$, then  there exists subset $S$ of $V(G)$ 
with the following properties:
\begin{enumerate}{
\item $\OMEGA_l(G\setminus [S,F])=\OMEGA_l(H\setminus [S,F])$.
\item $S\supseteq \{v\in V(G):d_H(v)> h_q(v)+d_F(v)\}$.
\item For each vertex $v$ of $S$, $d_H(v)\ge h_q(v)+d_F(v)$.
}\end{enumerate}
}\end{thm}
\begin{proof}
{Apply  the same arguments of Theorems~\ref{thm:preliminary:generalization} with replacing $h_q(v)$ instead of $h(v)$.
}\end{proof}
\begin{proofF}{
\noindent
\hspace*{-4mm}
\textbf{ Theorem~\ref{thm:final:gen}.}
First, define  $h_q(v)$ and  $\lambda_q$ as with $h(v)$  and $\lambda$ 
 in the proof of Theorem~\ref{thm:Second-Generalization} by replacing $\eta_q$ instead of $\eta$,  where $1\le q\le p$.
Next, for a fixed graph  $H\in \Gamma_p\subseteq \cdots\subseteq\Gamma_1$,  
show that  $te(H,h_q+d_F)\le t_q$,
 for any  $q$ with $1\le q\le p$,
by  repeatedly applying Theorem~\ref{thm:preliminary:generalization:excesses} and
using
the same arguments in the proof of  Theorem~\ref{thm:Second-Generalization}.
}\end{proofF}
%
%
%
%=================================
%
%
\section{Packing spanning partition-connected subgraphs}
\label{sec:packing}
In this section, we investigate  edge-decomposition  of  highly partition-connected graphs  into partition-connected spanning subgraphs.
For this purpose, we  first form the following lemma, which provides a generalization for  Lemma 3.5.3 in~\cite{MR1743598}.
\begin{thm}\label{thm:generalized:D}
{Let $G$ be a graph and let $l_1, l_2,\ldots, l_m$ be $m$  intersecting supermodular  subadditive integer-valued functions on subsets of $V(G)$.  
 If $F_1, \ldots, F_m$ is a family of  edge-disjoint spanning subgraphs of $G$ 
with the maximum $|E(F_1 \cup \cdots \cup F_m)|$
such that every  graph $F_i$ is $l_i$-sparse,
 then there is a partition $P$ of $V(G)$ such that
 there is no edges in $E(G)\setminus E(F_1 \cup \cdots \cup F_m)$ joining different parts of $P$, 
and also  for each $i$ with $1 \le i \le m$ and  every $A\in P$, the graph $F_i[A]$ is $l_i$-partition-connected.
}\end{thm}
\begin{proof}
{Define $F=(F_1,\ldots, F_m)$.
 Let $\mathcal{A}$ be the set of all $m$-tuples $\mathcal{F}=(\mathcal{F}_1,\ldots, \mathcal{F}_m)$ 
with the maximum  $|E(\mathcal{F})|$
 such that $\mathcal{F}_1,\ldots, \mathcal{F}_m$ are edge-disjoint spanning subgraphs of $G$ 
and every $\mathcal{F}_i$ is $l_i$-sparse,
where $E(\mathcal{F}) =E(\mathcal{F}_1 \cup \cdots \cup \mathcal{F}_m)$.
Note that if  $e\in E(G)\setminus E(\mathcal{F})$, then every graph $\mathcal{F}_i+e$ is not $l_i$-sparse; 
otherwise, we replace $\mathcal{F}_i$ by $\mathcal{F}_i+e$ in $\mathcal{F}$, 
which contradicts maximality of $|E(\mathcal{F})|$.
Thus both ends of $e$ lie an $l_i$-partition-connected subgraph of $\mathcal{F}_i$.
Let $Q_i$ be the  $l_i$-partition-connected subgraph of $\mathcal{F}_i$ 
including both ends of $e$ 
with minimum number of vertices.
Let $e'\in Q_i$.
Define $\mathcal{F}'_i=\mathcal{F}_i-e'+e$, 
and $\mathcal{F}'_j=\mathcal{F}_j$ for all $j$ with $j\neq i$.
According to Proposition~\ref{prop:xGy-exchange}, 
the graph $\mathcal{F}'_i$ is again $l_i$-sparse and so 
$\mathcal{F}'=(\mathcal{F}'_1,\ldots, \mathcal{F}'_m) \in \mathcal{A}$.
We say that $\mathcal{F}'$ is obtained from $\mathcal{F}$ by replacing a pair of edges.
Let $\mathcal{A}_0$ be the set of all $m$-tuples $\mathcal{F}$ in $\mathcal{A}$ 
which can be obtained from $F$ by a series of edge replacements.
Let $G_0$ be the spanning subgraph of $G$ with
$$E(G_0)=\bigcup_{\mathcal{F}\in  \mathcal{A}_0} (E(G)\setminus E(\mathcal{F})).$$
Now, we prove the following claim.
%
%===================================	Claim
\vspace{2mm}\\
{\bf Claim.} 
Let  $\mathcal{F}=(\mathcal{F}_1,\ldots, \mathcal{F}_m) \in \mathcal{A}_0$ and 
assume  that 
$\mathcal{F}'=(\mathcal{F}'_1,\ldots, \mathcal{F}'_m)$  is obtained from $\mathcal{F}$ by replacing a pair of edges.
If $x$ and $y$ are two vertices in an  $l_i$-partition-connected subgraph  of $\mathcal{F}'_i \cap G_0$, 
then $x$ and $y$ are also  in an $l_i$-partition-connected subgraph of $\mathcal{F}_i \cap G_0$, where $1 \le i \le m$.
\vspace{2mm}\\
%===================================
{\bf Proof of Claim.} 
Let $e'$ be the new edge in $E(\mathcal{F}')\setminus E(\mathcal{F})$.
Define $Q'_i$ to be the minimal $l_i$-partition-connected subgraph  of $\mathcal{F}'_i \cap G_0$ including $x$ and $y$.
We may assume that $e'\in E(Q'_i)$; 
otherwise, $E(Q'_i) \subseteq E(\mathcal{F}_i) \cap E(G_0)$ and the proof can easily be completed.
Since $e' \in E(\mathcal{F}')\setminus E(\mathcal{F})$, both ends  of $e'$
  must lie in an $l_i$-partition-connected subgraph   of $\mathcal{F}_i $.
%.
Define $Q_i$ to be the minimal $l_i$-partition-connected subgraph  of $\mathcal{F}_i$
including both ends of $e'$.
By Proposition~\ref{prop:xGy-exchange}, for every edge $e\in E(Q_i)$,
 the graph $\mathcal{F}_i-e+e'$ remains $l_i$-sparse, which can imply that $E(Q_i)\subseteq E(G_0)$.
Define $Q=(Q_i\cup Q'_i)-e'$.
Note that  $Q$ includes  $x$ and $y$, and also $E(Q)\subseteq E(G_0)\cap E(\mathcal{F}_i)$.
Since $Q/V(Q_i)$ and $Q[V(Q_i)]$ are $l_i$-partition-connected, by Proposition~\ref{prop:deducing}, the graph $Q$ itself must be   $l_i$-partition-connected.
Hence  the claim holds.

Define $P$ to be the partition of $V(G)$ obtained from the components of $G_0$.
Let $i \in \{1,\ldots, m\}$,
let   $C_0$ be a  component of $G_0$, 
and let $xy\in E(C_0)$.  
By the definition of $G_0$, 
there is no edges in $E(G)\setminus E(F_1 \cup \cdots \cup F_m)$ joining different parts of $P$, and also
 there are some $m$-tuples $\mathcal{F}^1,\ldots, \mathcal{F}^n$  in
 $\mathcal{A}_0$ such that
 $xy \in E(G)\setminus E(\mathcal{F}^n)$, 
 $F = \mathcal{F}^1$,
and every $\mathcal{F}^k$
 can be obtained from $\mathcal{F}^{k-1}$ by replacing a pair of edges,  where $1 < k \le n$.
As we stated above,  $x$ and $y$ must lie   in an $l_i$-partition-connected  subgraph  of $\mathcal{F}^n_i$.
Let $Q'_i$ be the  minimal $l_i$-partition-connected subgraph of $\mathcal{F}^n_i$
 including $x$ and $y$.
By Proposition~\ref{prop:xGy-exchange}, 
for every edge $e\in E(Q'_i)$, the graph  $\mathcal{F}^n_i-e+xy$ remains $l_i$-sparse,  which can imply
$E(Q'_i)\subseteq E(G_0)$. 
Thus $x$ and $y$ must also lie  in an $l_i$-partition-connected subgraph  of $\mathcal{F}^n_i\cap G_0$.
By repeatedly applying the above-mentioned claim, 
 one can conclude  that 
$x$ and $y$ lie  in an  $l_i$-partition-connected subgraph of $F_i \cap G_0$.
Let $Q_i$ be the  minimal $l_i$-partition-connected subgraph of $F_i$  including $x$ and $y$
so that $E(Q_i)\subseteq E(G_0)$. 
Since $l$ is subadditive, Proposition~\ref{prop:Q:subadditive} implies that $d_{Q_i}(A)\ge 1$,
for every vertex set $A$ with $\{x,y\} \subseteq A\subsetneq V(Q_i)$.
Since $C_0$ is connected,  we must have $V(Q_i)\subseteq V(C_0)$.
In other words,  for every $xy\in E(C_0)$, there is an $l_i$-partition-connected subgraph of $F_i\cap C_0$ including $x$ and $y$.
Since  $C_0$ is connected,  all vertices of $C_0$ must lie in an $l_i$-partition-connected subgraph of $F_i\cap C_0$.
Thus the graph $F_i [V(C_0)]$ itself must be   $l_i$-partition-connected.
Hence the proof is  completed.
}\end{proof}
The following theorem generalizes the well-known result of Nash-Williams~\cite{MR0133253} and Tutte~\cite{MR0140438}.
\begin{thm}\label{thm:main:partition-connected}
{Let $G$ be a graph and let $l_1, l_2,\ldots, l_m$ be $m$  intersecting supermodular subadditive integer-valued functions on subsets of $V(G)$.  
 If  $G$ is $(l_1+\cdots +l_m)$-partition-connected, then  it can be decomposed into $m$ edge-disjoint spanning subgraphs $H_1,\ldots, H_m$ such that every graph $H_i$ is $l_i$-partition-connected.
}\end{thm}
\begin{proof}
{Let $F_1, \ldots, F_m$ be a family of  edge-disjoint spanning subgraphs of $G$ 
with the maximum $|E(F)|$
such that every  graph $F_i$ is $l_i$-sparse, where $F=F_1 \cup \cdots \cup F_m$.
Let $P$ be a partition of $V(G)$ with the properties described in Theorem~\ref{thm:generalized:D}.
Since for  every $A\in P$, the induced subgraph $F_i[A]$ is $l_i$-partition-connected, 
we must have $e_{F_i}(A)=\sum_{v\in A}l_i(v)-l_i(A)$.
Define $l=l_1+\cdots +l_m$. 
By the assumption, $e_G(P)\ge \sum_{A\in P} l(A)-l(G)$.
Since $e_F(P)=e_G(P)$,  we  have
$$|E(F)| = e_F(P)+\sum_{A\in P} e_F(A) \ge \sum_{A\in P}l(A)-l(G)
+\sum_{A\in P}(\sum_{v\in A}l(v)-l(A))=\sum_{v\in V(G)}l(v)-l(G).$$
On the other hand,  
$$|E(F)| =\sum_{1\le i\le m}|E(F_i)|  \le \sum_{1\le i\le m}(\sum_{v\in V(G)}l_i(v)-l_i(G))=\sum_{v\in V(G)}l(v)-l(G).$$
Therefore, for every graph $F_i$, the equality  $|E(F_i)|=\sum_{v\in V(G)}l_i(v) -l_i(G)$ must be hold,
which implies that $F_i$ is $l_i$-partition-connected. 
This  can  complete the proof.
}\end{proof}
\begin{cor}{\rm (\cite{MR1411692}, see~\cite[Theorem 10.5.9]{MR2848535})}
{Every $l_{p+m,m}$-partition-connected graph has a packing of $m$ spanning trees and $p$ spanning $l_{1,0}$-partition-connected subgraphs,  where $l_{n,m}$ denotes the set function that is $n$ on vertices and is $m$ on vertex sets with at least two vertices.
}\end{cor}
In the following, we give an alternative  proof for a special case of Theorem~\ref{thm:main:result}.
\begin{cor}
{Let $G$ be a graph and let $l$ be an intersecting supermodular  subadditive nonnegative integer-valued function on subsets of $V(G)$.
If $G$ is $2l$-edge-connected, then it has
a spanning  $l$-partition-connected subgraph $H$ such that for each vertex $v$,
$$d_H(v) \le \lceil  \frac{d_G(v)}{2}\rceil+l(v).$$
Furthermore, for a given  arbitrary  vertex $u$ the upper bound can be reduced to  $ \lfloor  \frac{d_G(u)}{2}\rfloor+l(u)-l(G)$.
}\end{cor}
\begin{proof}
{Define $\ell(u)=\lceil d_G(u)/2  \rceil -l(u)+l(G)$,  and  $\ell(v)=\lfloor d_G(v)/2  \rfloor -l(v)$, 
for each vertex  $v$ with $v\neq u$ so that $d_G(u)\ge 2l(u)+2\ell(u)-2l(G)-1$ and $d_G(v)\ge 2l(v)+2\ell(v)$. 
Define $\ell(A)=0$ for every vertex set $A$ with $|A| \ge 2$.
Let $P$ be a partition of $V(G)$.  
By the assumption,
$$\sum_{A \in P} d_G(A) \ge
 \sum_{A\in P, |A|\ge 2} 2l(A)+ \sum_{A\in P, |A|= 1}d_G(A)\ge  
 \sum_{A\in P} (2l(A)+2\ell(A))-2l(G)-1.$$
which implies that
$$ e_G(P)= \frac{1}{2}\sum_{A \in P} d_G(A) \ge  \sum_{A\in P} (l(A)+\ell(A))-l(G)-\ell(G).$$
Thus $G$ is $(l+\ell)$-partition-connected.
By Theorem~\ref{thm:main:partition-connected}, 
the  graph $G$ can be decomposed into an $l$-partition-connected spanning   subgraph
  $H$  and an $\ell$-partition-connected spanning   subgraph  $H'$.
For each vertex $v$, we must have 
$d_{H'}(v) \ge \ell(G-v)+\ell(v)-\ell(G)=\ell(v)$.
This implies that $d_H(v)= d_G(v)-d_{H'}(v)\le \lceil d_G(v)/2\rceil+l(v)$.
Likewise, $d_H(u)= d_G(u)-d_{H'}(u)\le \lfloor d_G(u)/2\rfloor+l(u)-l(G)$.
Hence the corollary is proved.
}\end{proof}
\begin{cor}\label{cor:orientations:dec}
{Let $G$ be a graph and let $\ell_1, \ell_2,\ldots, \ell_m$ be $m$ nonincreasing intersecting supermodular nonnegative integer-valued functions on subsets of $V(G)$ with $\ell_1(G)= \cdots=\ell_m(G)=0$ .  
 If  $G$ is $(2\ell_1+\cdots +2\ell_m)$-edge-connected, then  it has an orientation and $m$ edge-disjoint spanning subdigraphs $H_1,\ldots, H_m$ such that every digraph $H_i$ is  $\ell_i$-arc-connected and for each vertex $v$,
$$d^+_G(v) \le \lceil  \frac{d_G(v)}{2}\rceil.$$
Furthermore, for a given  arbitrary  vertex $u$ the upper bound can be reduced to  $ \lfloor  \frac{d_G(u)}{2}\rfloor$.
}\end{cor}
\begin{proof}
{Define $\ell_0(u)=\lceil d_G(u)/2  \rceil -\ell(u)$,  and  
$\ell_0(v)=\lfloor d_G(v)/2  \rfloor -\ell(v)$ for each vertex $v$ with $v\neq u$, where  $\ell=\ell_1+\cdots +\ell_m$.
Define $\ell_0(A)=0$ for every vertex set $A$ with $|A| \ge 2$.
Let $P$ be a partition of $V(G)$.  
By the assumption,
$$\sum_{A \in P} d_G(A) \ge
 \sum_{A\in P, |A|\ge 2} 2\ell(A)+ \sum_{A\in P, |A|= 1}d_G(A)\ge  
 \sum_{A\in P} (2\ell(A)+2\ell_0(A))-1.$$
which implies that
$$ e_G(P)= \frac{1}{2}\sum_{A \in P} d_G(A) \ge  \sum_{A\in P} (\ell(A)+\ell_0(A))-\ell(G)-\ell_0(G).$$
Thus $G$ is $(\ell+\ell_0)$-partition-connected.
By Theorem~\ref{thm:main:partition-connected}, 
the  graph $G$ can be decomposed into $m+1$ edge-disjoint  spanning   subgraphs $H_0,\ldots, H_m$
  such that every $H_i$ is $\ell_i$-partition-connected.
By Lemma~\ref{lem:Frank}, every $H_i$ has an $\ell_i$-arc-connected orientation.
Consider the  orientation of  $G$ obtained from these orientations.
For each vertex $v$, we must have 
$d^+_{G}(v)\le d_{G}(v) -\sum_{0\le i\le k}d^-_{H_i}(v)  \le \lceil \frac{d_G(v)}{2}\rceil$.
Likewise, $d^+_{G}(u)\le  d_{G}(u) -\sum_{0\le i\le k}d^-_{H_i}(u) \le \lfloor d_G(u)/2\rfloor$.
Hence the corollary is proved.
}\end{proof}
\begin{cor}
{Let $G$ be a graph and let $l_1, l_2,\ldots, l_m$ be $m$ nonincreasing intersecting supermodular nonnegative integer-valued functions on subsets of $V(G)$ and let $r_1,\ldots, r_m$ be $m$ nonnegative integer-valued functions on $V(G)$ with $l_i(G)=\sum_{v\in V(G)}r_i(v)$.  
 If  $G$ is $(2l_1+\cdots +2l_m)$-edge-connected, then  it has an orientation and $m$ edge-disjoint spanning subdigraphs $H_1,\ldots, H_m$ such that every digraph $H_i$ is $r_i$-rooted $l_i$-arc-connected and  for each vertex $v$,
$$d^+_G(v) \le \lceil  \frac{d_G(v)}{2}\rceil.$$
Furthermore, for a given  arbitrary  vertex $u$ the upper bound can be reduced to  $ \lfloor  \frac{d_G(u)}{2}\rfloor$.
}\end{cor}
\begin{proof}
{Apply Corollary~\ref{cor:orientations:dec} with $\ell_i=l_i-r_i$, where $r_i(A)=\sum_{v\in A}r_i(v)$ for every vertex set $A$.
}\end{proof}
%
%
%
%==============================================================================
%
%
%
%
%
%
%
%
%==============================================================================
%
\section{Packing spanning partition-connected sub-hypergraphs}
%
%\section{Hypergraph versions}
\label{sec:hypergraphs}
In this subsection, we shall develop several results in this paper to hypergraphs in the same way.
Before doing so, we introduce  the needed definitions and notations for hypergraphs.
\subsection{Definitions}
%
%\subsection{Definitions}
 Let $\mathcal{H}$ be a hypergraph (possibly with repetition of hyperedges). 
The rank of  $\mathcal{H}$  is the maximum size of its hyperedges.
The vertex set and  the hyperedge set  of $\mathcal{H}$ are denoted by $V(\mathcal{H})$ and $E(\mathcal{H})$, respectively. 
The degree $d_\mathcal{H}(v)$ of a vertex $v$ is the number of hyperedges of $\mathcal{H}$ including  $v$.
For a set $X\subseteq V(\mathcal{H})$, 
we denote by $\mathcal{H}[X]$ the induced  sub-hypergraph of $\mathcal{H}$  with the vertex set $X$  containing
precisely those hyperedges  $Z$
of $\mathcal{H}$ with $Z\subseteq X$.
We also denote by $\mathcal{H}/X$  the hypergraph obtained from $\mathcal{H}$ by contracting $X$ into a single vertex $u$ 
and replacing each  hyperedge $Z$ with  $Z\cap X\neq \emptyset $ by $(Z\setminus  X)\cup \{u\}$.
A spanning sub-hypergraph $F$ is called {\bf $l$-sparse}, if for all vertex sets $A$, $e_F(A)\le \sum_{v\in A} l(v)-l(A)$,
where $e_F(A)$ denotes the number of hyperedges $Z$ of $F$  with $Z\subseteq A$.
Likewise,  the hypergraph $\mathcal{H}$ is called {\bf $l$-partition-connected}, 
if for every partition $P$ of $V(\mathcal{H})$, 
$e_\mathcal{H}(P)\ge \sum_{A\in P}l(A)-l(\mathcal{H})$, 
where $e_\mathcal{H}(P)$ denotes the number of hyperedges of $\mathcal{H}$ joining different parts of $P$.
Note that if $l$ is intersecting supermodular, then 
the vertex set of $\mathcal{H}$ can be expressed uniquely (up to order) as a disjoint union of vertex sets of some
induced $l$-partition-connected sub-hypergraphs.
These sub-hypergraphs are called the $l$-partition-connected components of $\mathcal{H}$.
To measure $l$-partition-connectivity of $\mathcal{H}$, we define the  parameter
$\OMEGA_l(\mathcal{H})=\sum_{A\in P} l(A)-e_\mathcal{H}(P),$ where 
 $P$ is the partition of $V(\mathcal{H})$ obtained from $l$-partition-connected components of $\mathcal{H}$.
It is not difficult to show that 
$\OMEGA_l(\mathcal{H})$ is  the maximum of all $\sum_{A\in P} l(A)-e_\mathcal{H}(P)$ taken over all
  partitions $P$ of $V(\mathcal{H})$.
The hypergraph $\mathcal{H}$ is said to be {\bf $l$-edge-connected},
 if for all nonempty proper vertex sets $A$ of $V(\mathcal{H})$, $d_\mathcal{H}(A) \ge l(A)$, 
where  $d_\mathcal{H}(A)$ 
denotes the number of hyperedges $Z$ of $\mathcal{H}$ with $Z\cap A\neq \emptyset$ and $Z\setminus A\neq \emptyset$.
For a vertex set $S$, we denote by $\sigma_\mathcal{H}(S)$ the sum of all $|Z\cap S|-1$ taken over all hyperedges $Z$  of $\mathcal{H}$ with $Z\cap S\neq \emptyset$.
Note that for graphs we have $\sigma_G(S) =e_G(S)$.
We call a hypergraph $\mathcal{H}$ directed, if for every hyperedge $Z$, a head vertex $u$ in $Z$ is specified; 
 other vertices of $Z-u$ are called the tails of $Z$.
For a vertex $v$, we denoted by $d^-_{\mathcal{H}}(v)$ the number of hyperedges with head  $v$ and denote by $d^+_{\mathcal{H}}(v)$ and the number of hyperedges with tail $v$. 
We say that a directed hypergraph   $\mathcal{H}$ is {\bf $l$-arc-connected}, if for every vertex set $A$, 
$d^-_{\mathcal{H}}(A)\ge l(A)$,
where $d^-_{\mathcal{H}}(A)$ denotes the number of hyperedges  $Z$  with head vertex in $A$ and 
$Z\setminus A\neq \emptyset $.
Likewise, $\mathcal{H}$ is called {\bf $r$-rooted $l$-arc-connected},
 if for every vertex set $A$, $d_\mathcal{H}^-(A)\ge l(A)-\sum_{v\in A}r(v)$, where
 $r$ is a nonnegative integer-valued on $V(\mathcal{H})$ with $l(\mathcal{H})=\sum_{v\in V(\mathcal{H})}r(v)$.
{\bf Trimming} a hyperedge $Z$ of size at least three is the operation that $Z$ is replaced  by a subset of it with size at least two, see~\cite{MR2848535}.
Trimming  a directed hyperedge $Z$ of size at least three with head  $u$  is the operation that $Z$ is replaced by a subset of it
including  $u$ with size at least two.
A trimmed (directed) hypergraph refers to a (directed) hypergraph which is obtained by a series of trimming operations.
 Throughout this article, all hypergraphs have hyperedges with size at least two.
%
%
%
%
%
%==============================================================================
%
%
%
%
%
%
%
%
%==============================================================================
%
%
\subsection{Basic tools}
For every vertex $v$ of a hypergraph $\mathcal{H}$, consider an induced $l$-partition-connected sub-hypergraph of 
$\mathcal{H}$ containing $v$ with
the maximal order. The following proposition shows that these sub-hypergraphs are unique and decompose the
vertex set of $\mathcal{H}$ when l is intersecting supermodular. 
The proofs of the results in this subsection are similar to whose graph versions. 
\begin{proposition}
{Let $\mathcal{H}$ be a hypergraph with  $X, Y\subseteq V(\mathcal{H})$ and  let $l$ be an   intersecting supermodular real function on 
%*
 subsets of $V(\mathcal{H})$.
If  $\mathcal{H}[X]$ and $\mathcal{H}[Y]$ are $l$-partition-connected and $X\cap Y\neq \emptyset$, 
then $\mathcal{H}[X\cup Y]$ is also $l$-partition-connected.
}\end{proposition}
\begin{proposition}\label{prop:hypergraph:contraction}
{Let $\mathcal{H}$ be a hypergraph with $ X\subseteq V(\mathcal{H})$ and  let $l$ be an intersecting supermodular  real function on subsets of $V(\mathcal{H})$.
If  $\mathcal{H}[X]$ and $\mathcal{H}/X$ are $l$-partition-connected,
then $\mathcal{H}$ is also $l$-partition-connected.
}\end{proposition}
\begin{lem}
{Let $\mathcal{H}$ be a hypergraph and let $l$ be a real  function on
%*
subsets of $V(\mathcal{H})$.
If $G$ is $l$-partition-connected and $P$ is a   partition of $V(\mathcal{H})$ with 
$$e_\mathcal{H}(P)=\sum_{A\in P}l(A)-l(\mathcal{H}),$$
then for any $A\in P$, the hypergraph $\mathcal{H}[A]$ is also $l$-partition-connected.
}\end{lem}
\begin{proposition}
{Let $\mathcal{H}$ be a hypergraph and let  $l$ be  an intersecting supermodular weakly subadditive integer-valued function on 
%*
subsets of $V(\mathcal{H})$.
If  $\mathcal{H}$ is minimally $l$-partition-connected, then
$$|E(\mathcal{H})|= \sum_{v\in V(\mathcal{H})} l(v)-l(\mathcal{H}).$$
}\end{proposition}
\begin{proposition}\label{prop:hypergraph:sparse}
{Let $F$ be an $l$-sparse hypergraph with $|E(F)| = \sum_{v\in V(F)}l(v)-l(F)$, where $l$ is a weakly subadditive real function on 
%*
subsets of $V(F)$. 
If $P$ is a partition of $V(F)$, then
$$e_F(P)\ge \sum_{A\in P}l(A)-l(F).$$
Furthermore, the equality holds only if for every $A\in P$, the hypergraph $F[A]$ is $l$-partition-connected.
}\end{proposition}
\begin{proposition}\label{prop:hypergraph:Q:subadditive}
{Let $F$ be an $l$-sparse hypergraph with $Y\subseteq  V(F)$, where  $l$ is  a weakly subadditive real function on subsets of $V(F)$. Let   $Q$ be an  $l$-partition-connected sub-hypergraph of $F$ with the minimum number of vertices including all vertices of  $Y$. 
If $l$ is element-subadditive, then for each $z\in V(Q)\setminus Y$, $d_Q(z)\ge 1$. 
Furthermore, if $l$ is subadditive, then for every vertex set $A$ with $Y\subseteq A\subsetneq V(Q)$, $d_Q(A)\ge 1$.
}\end{proposition}
\begin{proposition}\label{prop:hypergraph:xGy-exchange}
{Let $\mathcal{H}$ be a hypergraph and let $l$ be an  intersecting supermodular  weakly subadditive  integer-valued function on 
subsets of $V(\mathcal{H})$. Let $F$ be an $l$-sparse spanning sub-hypergraph  of $\mathcal{H}$.  
If  $Z'\in E(\mathcal{H})\setminus E(F)$ and $Q$ is an  $l$-partition-connected sub-hypergraph of $F$ including all vertices of $Z'$ with the minimum number of vertices,  then for every  $Z \in E(Q)$, the hypergraph $F-Z+Z'$ remains $l$-sparse.
}\end{proposition}
%
%==============================================================================
%
%
%
%
%
%
%
%
%==============================================================================
%
%
\subsection{Packing spanning partition-connected sub-hypergraphs}
The following theorem provides an extension for  Theorem~\ref{thm:generalized:D} on hypergraphs.
\begin{thm}\label{thm:hypergraph:generalized:D}
{Let $\mathcal{H}$ be a hypergraph and let $l_1, l_2,\ldots, l_m$ be  $m$ intersecting supermodular  subadditive integer-valued functions on subsets of $V(\mathcal{H})$.  
 If $F_1, \ldots, F_m$ is a family of  edge-disjoint spanning sub-hypergraphs of $\mathcal{H}$ 
with the maximum $|E(F_1 \cup \cdots \cup F_m)|$
such that every  hypergraph $F_i$ is $l_i$-sparse,
 then there is a partition $P$ of $V(\mathcal{H})$ such that
 there is no hyperedges in $E(\mathcal{H})\setminus E(F_1 \cup \cdots \cup F_m)$ joining different parts of $P$, 
and also  for each $i$ with $1 \le i \le m$ and  every $A\in P$, the hypergraph $F_i[A]$ is $l_i$-partition-connected.
}\end{thm}

\begin{proof}
{Define $F=(F_1,\ldots, F_m)$.
 Let $\mathcal{A}$ be the set of all $m$-tuples $\mathcal{F}=(\mathcal{F}_1,\ldots, \mathcal{F}_m)$ 
with the maximum  $|E(\mathcal{F})|$
 such that $\mathcal{F}_1,\ldots, \mathcal{F}_m$ are edge-disjoint spanning sub-hypergraphs of $\mathcal{H}$ 
and every $\mathcal{F}_i$ is $l_i$-sparse,
where $E(\mathcal{F}) =E(\mathcal{F}_1 \cup \cdots \cup \mathcal{F}_m)$.
Note that if  $Z\in E(\mathcal{H})\setminus E(\mathcal{F})$, then every hypergraph $\mathcal{F}_i+Z$ is not $l_i$-sparse; 
otherwise, we replace $\mathcal{F}_i$ by $\mathcal{F}_i+Z$ in $\mathcal{F}$, 
which contradicts maximality of $|E(\mathcal{F})|$.
Thus all vertices of $Z$ lie in an $l_i$-partition-connected subgraph of $\mathcal{F}_i$.
Let $Q_i$ be the  $l_i$-partition-connected sub-hypergraph of $\mathcal{F}_i$ 
including  all vertices of $Z$
with minimum number of vertices.
Let $Z'\in Q_i$.
Define $\mathcal{F}'_i=\mathcal{F}_i-Z'+Z$, 
and $\mathcal{F}'_j=\mathcal{F}_j$ for all $j$ with $j\neq i$.
According to Proposition~\ref{prop:hypergraph:xGy-exchange}, 
the hypergraph $\mathcal{F}'_i$ is again $l_i$-sparse and so 
$\mathcal{F}'=(\mathcal{F}'_1,\ldots, \mathcal{F}'_m) \in \mathcal{A}$.
We say that $\mathcal{F}'$ is obtained from $\mathcal{F}$ by replacing a pair of hyperedges.
Let $\mathcal{A}_0$ be the set of all $m$-tuples $\mathcal{F}$ in $\mathcal{A}$ 
which can be obtained from $F$ by a series of hyperedge replacements.
Let $\mathcal{H}_0$ be the spanning sub-hypergraph of $\mathcal{H}$ with
$$E(\mathcal{H}_0)=\bigcup_{\mathcal{F}\in  \mathcal{A}_0} (E(\mathcal{H})\setminus E(\mathcal{F})).$$
Now, we prove the following claim.
%
%===================================	Claim
\vspace{2mm}\\
{\bf Claim.} 
Let  $\mathcal{F}=(\mathcal{F}_1,\ldots, \mathcal{F}_m) \in \mathcal{A}_0$ and 
assume  that 
$\mathcal{F}'=(\mathcal{F}'_1,\ldots, \mathcal{F}'_m)$  is obtained from $\mathcal{F}$ by replacing a pair of hyperedges.
If all vertices of a given vertex set $Y$ lie in an  $l_i$-partition-connected sub-hypergraph  of $\mathcal{F}'_i \cap \mathcal{H}_0$, 
then those vertices  also lie  in an $l_i$-partition-connected sub-hypergraph of $\mathcal{F}_i \cap \mathcal{H}_0$, where $1 \le i \le m$.
\vspace{2mm}\\
%===================================
{\bf Proof of Claim.} 
Let $Z'$ be the new hyperedge in $E(\mathcal{F}')\setminus E(\mathcal{F})$.
Define $Q'_i$ to be the minimal $l_i$-partition-connected sub-hypergraph  of $\mathcal{F}'_i \cap  \mathcal{H}_0$ including all vertices of $Y$.
We may assume that $Z'\in E(Q'_i)$; 
otherwise, $E(Q'_i) \subseteq E(\mathcal{F}_i) \cap E(\mathcal{H}_0)$ and the proof can easily be completed.
Since $Z' \in E(\mathcal{F}')\setminus E(\mathcal{F})$, all vertices of $Z'$
  must lie in an $l_i$-partition-connected  sub-hypergraph   of $\mathcal{F}_i $.
%.
Define $Q_i$ to be the minimal $l_i$-partition-connected  sub-hypergraph  of $\mathcal{F}_i$
including all vertices of $Z'$.
By Proposition~\ref{prop:hypergraph:xGy-exchange}, for every hyperedge $Z\in E(Q_i)$,
 the hypergraph $\mathcal{F}_i-Z+Z'$ remains $l_i$-sparse, which can imply that $E(Q_i)\subseteq E(\mathcal{H}_0)$.
Define $Q=(Q_i\cup Q'_i)-Z'$.
Note that  $Q$ includes  all vertices of $Y$, and also $E(Q)\subseteq E(\mathcal{H}_0)\cap E(\mathcal{F}_i)$.
Since $Q/V(Q_i)$ and $Q[V(Q_i)]$ are $l_i$-partition-connected, by Proposition~\ref{prop:hypergraph:contraction}, the hypergraph $Q$ itself must be   $l_i$-partition-connected.
Hence  the claim holds.

Define $P$ to be the partition of $V(\mathcal{H})$ obtained from the components of $\mathcal{H}_0$.
Let $i \in \{1,\ldots, m\}$,
let   $C_0$ be a  component of $\mathcal{H}_0$, 
and let $Y\in E(C_0)$.  
By the definition of $\mathcal{H}_0$, 
there is no hyperedges in $E(\mathcal{H})\setminus E(F_1 \cup \cdots \cup F_m)$ joining different parts of $P$, and also
 there are some $m$-tuples $\mathcal{F}^1,\ldots, \mathcal{F}^n$  in
 $\mathcal{A}_0$ such that
 $Y \in E(\mathcal{H})\setminus E(\mathcal{F}^n)$, 
 $F = \mathcal{F}^1$,
and every $\mathcal{F}^k$
 can be obtained from $\mathcal{F}^{k-1}$ by replacing a pair of hyperedges,  where $1 < k \le n$.
As we stated above,  all vertices of $Y$ must lie   in an $l_i$-partition-connected   sub-hypergraph  of $\mathcal{F}^n_i$.
Let $Q'_i$ be the  minimal $l_i$-partition-connected  sub-hypergraph of $\mathcal{F}^n_i$
 including all vertices of $Y$.
By Proposition~\ref{prop:hypergraph:xGy-exchange}, 
for every hyperedge $Z\in E(Q'_i)$, the hypergraph  $\mathcal{F}^n_i-Z+Y$ remains $l_i$-sparse,  which can imply
$E(Q'_i)\subseteq E(\mathcal{H}_0)$. 
Thus all vertices of $Y$ must also lie  in an $l_i$-partition-connected subgraph  of $\mathcal{F}^n_i\cap \mathcal{H}_0$.
By repeatedly applying the above-mentioned claim, 
 one can conclude  that 
all vertices of $Y$ lie  in an  $l_i$-partition-connected  sub-hypergraph of $F_i \cap \mathcal{H}_0$.
Let $Q_i$ be the  minimal $l_i$-partition-connected  sub-hypergraph of $F_i$  including all vertices of $Y$
so that $E(Q_i)\subseteq E(\mathcal{H}_0)$. 
Since $l$ is subadditive, Proposition~\ref{prop:hypergraph:Q:subadditive} implies that $d_{Q_i}(A)\ge 1$,
for every vertex set $A$ with $Y \subseteq A\subsetneq V(Q_i)$.
Since $C_0$ is connected,  we must have $V(Q_i)\subseteq V(C_0)$.
In other words,  for every  $Y\in E(C_0)$, there is an $l_i$-partition-connected  sub-hypergraph of $F_i\cap C_0$ including all vertices of $Y$.
Since  $C_0$ is connected,  all vertices of $C_0$ must lie in an $l_i$-partition-connected  sub-hypergraph of $F_i\cap C_0$.
Thus the hypergraph $F_i [V(C_0)]$ itself must be   $l_i$-partition-connected.
Hence the proof is  completed.
}\end{proof}
The following theorem generalizes a  result   in~\cite[Theorem 2.8]{MR2021107} due to Frank, Kir\'aly, and Kriesell (2003).

\begin{thm}\label{thm:main:partition-connected:hypergraph}
{Let $\mathcal{H}$ be a hypergraph and let $l_1, l_2,\ldots, l_m$ be $m$  intersecting supermodular subadditive integer-valued functions on subsets of $V(\mathcal{H})$.  
 If $\mathcal{H}$ is $(l_1+\cdots +l_m)$-partition-connected,  then it can be decomposed into $m$ edge-disjoint spanning sub-hypergraphs $H_1,\ldots, H_m$ such that every hypergraph $H_i$ is $l_i$-partition-connected.
}\end{thm}
\begin{proof}
{Let $F_1, \ldots, F_m$ be a family of  edge-disjoint spanning sub-hypergraphs of $\mathcal{H}$ 
with the maximum $|E(F)|$
such that every  hypergraph $F_i$ is $l_i$-sparse, where $F=F_1 \cup \cdots \cup F_m$.
Let $P$ be a partition of $V(\mathcal{H})$ with the properties described in Theorem~\ref{thm:hypergraph:generalized:D}.
Since for  every $A\in P$, the induced sub-hypergraph $F_i[A]$ is $l_i$-partition-connected, 
we must have $e_{F_i}(A)=\sum_{v\in A}l_i(v)-l_i(A)$.
Define $l=l_1+\cdots +l_m$. 
By the assumption, $e_\mathcal{H}(P)\ge \sum_{A\in P} l(A)-l(\mathcal{H})$.
Since $e_F(P)=e_\mathcal{H}(P)$,  we  have
$$|E(F)| = e_F(P)+\sum_{A\in P} e_F(A) 
\ge \sum_{A\in P}l(A)-l(\mathcal{H})+\sum_{A\in P}(\sum_{v\in A}l(v)-l(A))=
\sum_{v\in V(\mathcal{H})}l(v)-l(\mathcal{H}).$$
On the other hand,  
$$|E(F)| =\sum_{1\le i\le m}|E(F_i)|  \le 
\sum_{1\le i\le m}(\sum_{v\in V(\mathcal{H})}l_i(v)-l_i(\mathcal{H}))=\sum_{v\in V(\mathcal{H})}l(v)-l(\mathcal{H}).$$
Therefore, for every hypergraph $F_i$, the equality  $|E(F_i)|=\sum_{v\in V(\mathcal{H})}l_i(v) -l_i(\mathcal{H})$ must be hold,
which implies that $F_i$ is $l_i$-partition-connected. 
This  can  complete the proof.
}\end{proof}
The following result provides an  improvement for Corollary~2.9 in~\cite{MR2021107}.
\begin{cor}\label{cor:hypergraph:rl-edge}
{Let $\mathcal{H}$ be a hypergraph  with  the rank $r$ and let $l$ be an intersecting supermodular  subadditive nonnegative integer-valued function on subsets of $V(\mathcal{H})$.
If $\mathcal{H}$ is $rl$-edge-connected, then it has 
an $l$-partition-connected spanning  sub-hypergraph $H$ such that for each vertex $v$,
$$d_H(v) \le \lceil  \frac{r-1}{r}d_\mathcal{H}(v)\rceil+l(v).$$
Furthermore, for a given arbitrary  vertex $u$ the upper bound can be reduced to 
 $ \lfloor   \frac{r-1}{r}d_\mathcal{H}(u)\rfloor+l(u)-l(\mathcal{H})$.
}\end{cor}
\begin{proof}
{Define $\ell(u)=\lceil d_\mathcal{H}(u)/r  \rceil -l(u)+l(\mathcal{H})$,  
and  $\ell(v)=\lfloor d_\mathcal{H}(v)/r  \rfloor -l(v)$, 
for each vertex   $v$ with $v\neq u$
 so that $d_\mathcal{H}(u)\ge rl(u)+r\ell(u)+rl(\mathcal{H})-(r-1)$ 
and $d_\mathcal{H}(v)\ge rl(v)+r\ell(v)$. 
Define $\ell(A)=0$ for every  vertex set $A$ with $|A| \ge 2$.
Let $P$ be a partition of $V(\mathcal{H})$.  
By the assumption,
$$\sum_{A \in P} d_\mathcal{H}(A) \ge
 \sum_{A\in P, |A|\ge 2} rl(A)+ \sum_{A\in P, |A|= 1}d_\mathcal{H}(A)\ge  
 \sum_{A\in P} (rl(A)+r\ell(A))-rl(\mathcal{H})-(r-1).$$
which implies that
$$ e_\mathcal{H}(P)\ge  \frac{1}{r}\sum_{A \in P} d_\mathcal{H}(A) \ge  \sum_{A\in P} (l(A)+\ell(A))
-l(\mathcal{H})-\ell(\mathcal{H}).$$
Thus $\mathcal{H}$ is $(l+\ell)$-partition-connected.
By Theorem~\ref{thm:main:partition-connected:hypergraph}, 
the  hypergraph $\mathcal{H}$ can be decomposed into an $l$-partition-connected  spanning  sub-hypergraph
  $H$  and an $\ell$-partition-connected spanning  sub-hypergraph  $H'$.
For each vertex $v$, we must have 
$d_{H'}(v) \ge \ell(\mathcal{H}-v)+\ell(v)-\ell(\mathcal{H})=\ell(v)$.
This implies that $d_H(v)= d_\mathcal{H}(v)-d_{H'}(v) \le \lceil \frac{r-1}{r} d_\mathcal{H}(v)\rceil+l(v)$.
Likewise, $d_H(u)= d_\mathcal{H}(u)-d_{H'}(u)\le  \lfloor \frac{r-1}{r}d_\mathcal{H}(u)\rfloor+l(u)-l(\mathcal{H})$.
Hence the theorem holds.
}\end{proof}
For every hypergraph $\mathcal{H}$, one may associate a  nonnegative set function $r$ such that for every vertex set $A$, 
$r(A)$ is the maximum  of  all $|A \setminus Z|+1$ taken over all hyperedges $Z$ with $|Z\cap A|\neq \emptyset$.
We call $r(A)$  the local rank of  $\mathcal{H}$ on $A$.
According this definition, the above-mentioned corollary could be refined to the following version.
\begin{cor}
{Let $\mathcal{H}$ be a hypergraph  with the  local rank function $r$ and let $l$ be an intersecting supermodular  subadditive nonnegative integer-valued function on subsets of $V(\mathcal{H})$.
If $\mathcal{H}$ is $rl$-edge-connected, then it has 
an $l$-partition-connected spanning sub-hypergraph $H$ such that for each vertex $v$,
$d_H(v) \le \lceil  \frac{r(v)-1}{r(v)}d_\mathcal{H}(v)\rceil+l(v).$
}\end{cor}
\begin{proof}
{Apply the same arguments of Corollary~\ref{cor:hypergraph:rl-edge} 
by replacing the inequality $ e_\mathcal{H}(P)\ge \sum_{A \in P}  \frac{1}{r(A)} d_\mathcal{H}(A).$
}\end{proof}
The following result can be proved similarly to whose graph version  and can also be formulated  in a rooted arc-connected version.
\begin{cor}
{Let $\mathcal{H}$ be a hypergraph with the rank $r$ and let $\ell_1, \ell_2,\ldots, \ell_m$ be $m$ nonincreasing intersecting supermodular nonnegative integer-valued functions on subsets of $V(\mathcal{H})$ with $\ell_1(\mathcal{H})= \cdots=\ell_m(\mathcal{H})=0$ .  
 If  $\mathcal{H}$ is $(r\ell_1+\cdots +r\ell_m)$-edge-connected, then  it has an orientation and $m$ edge-disjoint spanning sub-hypergraphs $H_1,\ldots, H_m$ such that every hypergraph $H_i$ is  $\ell_i$-arc-connected and for each vertex $v$,
$$d^+_\mathcal{H}(v) \le \lceil  \frac{r-1}{r}d_\mathcal{H}(v)\rceil.$$
Furthermore, for a given  arbitrary  vertex $u$ the upper bound can be reduced to 
 $ \lfloor   \frac{r-1}{r}d_\mathcal{H}(u)\rfloor$.
}\end{cor}
\subsection{Trimming hypergraphs and preserving partition-connectivity}
As we observed in the previous subsection, the proof of Theorem~\ref{thm:main:partition-connected:hypergraph} follows from the same arguments of whose graph version.
In fact,  Theorem~\ref{thm:main:partition-connected:hypergraph}  can easily be derived from whose graph version, 
using the following generalization of Theorem 9.4.5 in~\cite{MR2848535}.
\begin{thm}
{Let $\mathcal{H}$ be a hypergraph and 
let  $l$ be  an intersecting supermodular weakly subadditive integer-valued function on subsets of $V(\mathcal{H})$.  
If $\mathcal{H}$ is $l$-partition-connected, then it can be trimmed to 
an $l$-partition-connected  graph.
}\end{thm}
We  show  below that the  operations can be done without removing  specified vertices  from hyperedges.
\begin{thm}\label{thm:making-graphs-hypergraphs}
{Let $\mathcal{H}$ be a directed hypergraph  and let   $l$ be  an intersecting supermodular weakly subadditive integer-valued function on subsets of $V(\mathcal{H})$.  
If $\mathcal{H}$ is $l$-partition-connected, then it can be trimmed to an $l$-partition-connected directed graph.
}\end{thm}
\begin{proof}
{By induction on the sum of all $|Z|-2$ taken over all hyperedges $Z$.
If this sum  is zero, then $\mathcal{H}$ itself is a graph.
So assume that a directed hyperedge  $Z$ with head $u$ has size at least three.
Let $x$ a vertex  of $Z\setminus \{u\}$.
If replacing $Z$ by $Z-x$ preserves partition-connectivity, then the proof follows by induction.
 Otherwise, there is a partition $P$ of $V(\mathcal{H})$ such that 
$e_{\mathcal{H}}(P)=\sum_{A\in P}l(A)-l(\mathcal{H})$ and $Z\setminus X =\{x\}$, for a vertex set $X \in P$.
Let  $y$ be a vertex of $Z\cap X\setminus \{u\}$.
Now, replace  $Z$ by $Z-y$ and call the resulting  hypergraph $\mathcal{H}'$.
According to this construction, we must have $\mathcal{H}'/X=\mathcal{H}/X$ and $\mathcal{H}[X]=\mathcal{H}'[X]$.
Since $\mathcal{H}'/X$ and $\mathcal{H}'[X]$ are  $l$-partition-connected, by Proposition~\ref{prop:hypergraph:contraction}, the hypergraph $\mathcal{H}'$ itself must be $l$-partition-connected.
Thus by the induction hypothesis the theorem can be hold.
}\end{proof}
The following theorem is a counterpart of Theorem~\ref{thm:making-graphs-hypergraphs}.
\begin{thm}
{Let $\mathcal{H}$ be a directed hypergraph  and let   $l$ be  an intersecting supermodular weakly subadditive integer-valued function on subsets of $V(\mathcal{H})$.  
If $\mathcal{H}$ is $l$-sparse, then it can be trimmed to an $l$-sparse directed graph.
}\end{thm}

\begin{proof}
{By induction on the sum of all $|Z|-2$ taken over all hyperedges $Z$.
If this sum  is zero, then $\mathcal{H}$ itself is a graph.
So assume that a directed hyperedge  $Z$ with head $u$ has size at least three.
Let $x$ and $y$ be two vertices of $Z\setminus \{u\}$.
If replacing $Z$ by $Z-x$ preserves sparse property,  then the proof follows by induction.
Otherwise,  there is a vertex set $X$ including $u$
 such that $Z\setminus X=\{x\}$ and  $e_{\mathcal{H}}(X)=\sum_{v\in X}l(v)-l(X)$.
Corresponding to $y$, there is a vertex set $Y$ including $u$ 
 such that $Z\setminus Y=\{y\}$ and  $e_{\mathcal{H}}(Y)=\sum_{v\in Y}l(v)-l(Y)$.
Note that   $Z$ is neither a subset of  $X$ nor a subset of $Y$. 
Thus $$e_{\mathcal{H}}(X\cup  Y)\ge e_{\mathcal{H}}(X)+e_{\mathcal{H}}(Y)-e_{\mathcal{H}}(X\cap Y)+1.$$
Since $l$ is intersecting supermodular, we must have 
$$e_{\mathcal{H}}(X\cup  Y) \ge \sum_{v\in X}l(v)-l(X)+\sum_{v\in Y}l(v)-l(Y)-
 \sum_{v\in X\cup Y}l(v)+l(X\cap Y)+1
> \sum_{v\in X\cup  Y}l(v)+l(X\cup Y).$$
This is a contradiction, as desired.
}\end{proof}
The following theorem generalizes Theorem 7.4.9 in~\cite{MR2848535}.
\begin{thm}
{Let $\mathcal{H}$ be a directed hypergraph and 
let  $\ell$ be  a positively intersecting supermodular integer-valued function on subsets of $V(\mathcal{H})$ with $\ell(\emptyset)=\ell(\mathcal{H})=0$.   
If $\mathcal{H}$ is $\ell$-arc-connected, then it can be trimmed to 
an $\ell$-arc-connected directed graph.
}\end{thm}
\begin{proof}
{By induction on the sum of all $|Z|-2$ taken over all hyperedges $Z$.
If this sum  is zero, then $\mathcal{H}$ itself is a graph.
So assume that a directed hyperedge  $Z$ with head $u$ has size at least three.
Let $x$ and $y$ be two vertices of $Z\setminus \{u\}$.
If replacing $Z$ by $Z-x$ preserves arc-connectivity, then the proof follows by induction.
Otherwise,  there is a vertex set $X$ including $u$ such that  $Z\setminus X=\{x\}$ and $\ell(X)=d^-_{\mathcal{H}}(X)>0$.
Corresponding to $y$, there is a vertex set $Y$ including $u$ 
such that  $Z\setminus Y=\{y\}$ and $\ell(Y)=d^-_{\mathcal{H}}(Y)> 0$.
Note that   $Z$ is neither a subset of  $X$ nor a subset of $Y$. 
Thus 
$$\ell(X)+\ell(Y) =
d^-_{\mathcal{H}}(X)+d^-_{\mathcal{H}}(Y) \ge
  d^-_{\mathcal{H}}(X\cup Y)+d^-_{\mathcal{H}}(X\cap Y) +1>
\ell(X\cup Y)+\ell(X\cap Y).$$
Since $\ell$ is intersecting supermodular, we arrive at a contradiction. 
}\end{proof}
\subsection{Spanning partition-connected sub-hypergraphs with restricted degrees}
The following theorem provides a generalization for Theorem~\ref{thm:alternative}  with a new   proof.
\begin{thm}\label{thm:hypergraph:toughness:degrees}
Let $\mathcal{H}$ be a hypergraph  and let $l$ be an intersecting supermodular subadditive  integer-valued function on subsets of $V(\mathcal{H})$.  Let  $h$ be an  integer-valued function on $V(\mathcal{H})$.
If  for all $S\subseteq V(\mathcal{H})$, 
$$\OMEGA_l(\mathcal{H}\setminus S)\le \sum_{v\in S}\big(h(v)-l(v)\big)+l(\mathcal{H})-\sigma_\mathcal{H}(S),$$
then  $\mathcal{H}$ has an $l$-partition-connected  spanning sub-hypergraph  $H$ such that for each vertex  $v$,
 $d_H(v)\le h(v).$
\end{thm}
\begin{proof}
{Define $\ell(v)=\max\{0,d_\mathcal{H}(v)-h(v)\}$ for each vertex  $v$, 
and define $\ell(A)=0$ for every vertex set $A$ with $|A|\ge 2$.
Let $P$ be a partition of $V(\mathcal{H})$.
Define $S$ to be the set of all vertices $v$ such that $\{v\}\in P$ 
and  $\ell(v)=d_\mathcal{H}(v)-h(v)$.
Also, define $\mathcal{P}$ to  be set of all vertex sets $A\in P$ such that $A\neq \{v\}$, when $v\in S$.
Note that for every $A\in \mathcal{P}$, $\ell(A)=0$.
By the assumption, 
$$\sum_{A\in \mathcal{P}}l(A)-e_{\mathcal{H} \setminus S}(\mathcal{P}) \le 
\OMEGA_l(\mathcal{H}\setminus S)\le 
\sum_{v\in S}\big(h(v)-l(v)\big)+l(\mathcal{H})-\sigma_\mathcal{H}(S).$$
%t
Since $e_\mathcal{H}(P)=\sum_{v\in S}d_\mathcal{H}(v)-\sigma_\mathcal{H}(S)+e_{\mathcal{H}\setminus S} (\mathcal{P})$, we must have 
$$ e_\mathcal{H}(P) 
\ge \sum_{A\in \mathcal{P}}l(A)+\sum_{v\in S}\big(d_\mathcal{H}(v) -h(v)+l(v)\big)-l(\mathcal{H}),$$
which implies that
$$e_\mathcal{H}(P)  \ge 
\sum_{A\in \mathcal{P}}l(A)+\sum_{v\in S} (\ell(v)+l(v))-l(\mathcal{H})=
\sum_{A\in P}(l(A)+\ell(A)) -l(\mathcal{H})-\ell(\mathcal{H}).$$
Thus $\mathcal{H}$ is $(l+\ell)$-partition-connected. 
By Theorem~\ref{thm:main:partition-connected:hypergraph}, 
the  hypergraph $\mathcal{H}$ can be decomposed into
 an $l$-partition-connected spanning   sub-hypergraph  $H$  and
 an  $\ell$-partition-connected  spanning  sub-hypergraph  $H'$.
For each vertex $v$, we must have 
$d_{H'}(v) \ge \ell(\mathcal{H}-v)+\ell(v)-\ell(\mathcal{H})=\ell(v)$.
This implies that $d_H(v)\le d_\mathcal{H}(v)-d_{H'}(v)\le h(v)$.
Hence the theorem is proved.
}\end{proof}
The following theorem   provides two upper bounds on $\OMEGA_l(\mathcal{H}\setminus S)$ 
depending on two parameters of connectivity of $\mathcal{H}$ and $d_{\mathcal{H}}(v)$  of the vertices $v$ in $S$,
 which generalizes Lemma~\ref{lem:high-edge-connectivity:Theta}.
\begin{thm}
{Let $\mathcal{H}$ be a hypergraph with the rank $r$,  
let $l$ be an intersecting supermodular real function on subsets   of $V(\mathcal{H})$,
and let $k$ be a positive real number. If $S\subseteq V(\mathcal{H})$, then
$$\OMEGA_l(\mathcal{H}\setminus S)\le
 \begin{cases}
\sum_{v\in S}\frac{r-1}{k}d_\mathcal{H}(v)\,-\frac{r}{k}\sigma_\mathcal{H}(S),	
&\text{when $\mathcal{H}$ is $kl$-edge-connected, $k\ge r$, and  $S\neq \emptyset$};\\ 
 \sum_{v\in S}\big(\frac{d_\mathcal{H}(v)}{k}-l(v)\big)+l(\mathcal{H})-\frac{1}{k}\sigma_\mathcal{H}(S),	
&\text{when $\mathcal{H}$ is $kl$-partition-connected and $k\ge 1$}.
\end {cases}$$
}\end{thm}
\begin{proof}
{Let $P$ be the partition of $V(\mathcal{H})\setminus S$ obtained from $l$-partition-connected components of $\mathcal{H}\setminus S$.
For every integer $i$ with $0\le i \le r$, denote by $c_i$ the number of hyperedges $Z$ with $|Z\cap S|=i$.
If $\mathcal{H}$ is $kl$-edge-connected and  $S\neq \emptyset$, for any $A\in P$,
there  are at least $kl(A)$ hyperedges $Z$ with 
 $Z\cap A \neq \emptyset$ and $Z\setminus A \neq \emptyset$. Thus
$$\sum_{A\in P}k l(A)-rc_0
\le  \sum_{1\le i \le r} (r-i)c_i =
 \sum_{1\le i \le k}(r-1)ic_i -\sum_{1\le i \le k}r(i-1)c_i=\sum_{v\in S}(r-1)d_{\mathcal{H}}(v)-r\sigma_{\mathcal{H}}(S),$$
which  implies that
$$\OMEGA_l(\mathcal{H}\setminus S)\le \sum_{A\in P} l(A)-\frac{r}{k}e_{\mathcal{H}\setminus S}(P) \le 
\sum_{v\in S}\frac{(r-1)}{k}d_{\mathcal{H}}(v)-\frac{r}{k}\sigma_{\mathcal{H}}(S).$$
When  $\mathcal{H}$ is $kl$-partition-connected and $k\ge 1$,  we have 
$$ \sum_{A\in P}kl(A)+\sum_{v\in S} kl(v)-kl(\mathcal{H}) 
\le
e_\mathcal{H}( P\cup \{\{v\}:v\in S\})
= 
\sum_{v\in S}d_\mathcal{H}(v)\,-\sigma_\mathcal{H}(S)+e_{\mathcal{H}\setminus S}(P),$$
and so
$$k\OMEGA_l(\mathcal{H}\setminus S)=  \sum_{A\in P}kl(A)-ke_{\mathcal{H}\setminus S}(P)\le 
 \sum_{A\in P}kl(A)-e_{\mathcal{H}\setminus S}(P)\le 
\sum_{v\in S}\big(d_\mathcal{H}(v)-kl(v)\big)+kl(\mathcal{H})-\sigma_\mathcal{H}(S).$$
These inequalities can complete the proof.
}\end{proof}
\subsection{An application to packing Steiner trees with restricted degrees}
The following theorem  is a strengthened version of  Theorem 3.1 in~\cite{MR2021107} and can be proved  in the same way,
by replacing the new  improved version of  Corollary~2.9 in~\cite{MR2021107}.
\begin{thm}
{Let $G$ be graph with $S\subseteq V(G)$, where $V(G)\setminus S$ is an independent set. 
If $G$ is $3m$-edge-connected in $S$,  
then it  has a spanning  subgraph  $H$ containing $m$ edge-disjoint Steiner trees spanning $S$
such that for each $v\in S$,
$d_H(v)\le \lceil \frac{2}{3}d_G(v) \rceil +m.$
}\end{thm}
%
%
%
%=============================================================================
%
%
%
%
%
%
%
%
%
%
%
%
%
%=============================================================================
%
%==============================================================================

%\bibliographystyle{siam}
%\bibliography{ref}
\end{document}